\documentclass[12pt, reqno]{amsart}
\usepackage[utf8]{inputenc}
\usepackage{amsmath}
\usepackage{amsthm}
\usepackage{amsfonts}
\usepackage{amssymb}
\usepackage{mathrsfs, mathtools, mathdots,comment}
\usepackage{diagbox}
\usepackage[T1]{fontenc}
\usepackage{amscd}
\usepackage{enumitem}
\usepackage[mathscr]{eucal}
\usepackage{cleveref}
\usepackage[top=25mm,bottom=25mm,left=27mm,right=27mm]{geometry}

\setenumerate[1]{label=\arabic*.}

\theoremstyle{plain}
\newtheorem{Th}{Theorem}
\newtheorem{Lemma}{Lemma}
\newtheorem{Cor}{Corollary}
\newtheorem{Prop}{Proposition}

\newtheorem{property}{Property}
\makeatletter
\newcommand{\settheoremtag}[1]{
	\let\oldtheproperty\theproperty
	\renewcommand{\theproperty}{#1}
	\g@addto@macro\endproperty{
		\addtocounter{property}{-1}
		\global\let\theproperty\oldtheproperty}
}
\makeatother

\newtheorem{Conj}{Conjecture}

\theoremstyle{definition}

\newtheorem{Def}{Definition}[section]
\newtheorem{Rem}{Remark}[section]

\newtheorem{Ex}{Example}[section]

\newcommand{\C}{\mathbb{C}}
\newcommand{\bQ}{\overline{\mathbb{Q}}}
\newcommand{\Z}{\mathbb{Z}}

\newcommand{\D}{\mathcal{D}}

\newcommand{\Gal}[1]{Gal(#1)}

\newcommand{\Proj}{\mathrm{Proj}}
\newcommand{\Frob}[1]{\sigma_{#1}}
\newcommand{\cond}[1]{cond(#1)}

\begin{document}
\title{On primitivity and vanishing of Dirichlet Series}

\author[A. Bharadwaj]{Abhishek Bharadwaj}

\address{Department of Mathematics, Queen's University, 48 University Ave. Jeffery Hall, Kingston, ON Canada K7L 3N6}
\email{atb4@queensu.ca}

\date{\today}

\subjclass[2010]{11J72, 11M99}
\begin{abstract}
For a rational valued periodic function, we associate a Dirichlet series and provide a new necessary and sufficient condition for the vanishing of this Dirichlet series specialized at positive integers. This question was initiated by Chowla, and carried out by Okada for a particular infinite sum. Our approach relies on the decomposition of the Dirichlet characters in terms of primitive characters. Using this, we find some new family of natural numbers for which a conjecture of Erd\"{o}s holds. We also characterize rational valued periodic functions for which the associated Dirichlet series vanishes at two different positive integers under some additional conditions.     
\end{abstract}
\maketitle
\section{Introduction}\label{SEC:INTRO}
In a written communication to Livingston \cite{LIVINGSTON}, Erd\"{o}s proposed the following conjecture: 
\begin{Conj}[Erd\"{o}s]\label{CONJ:ERDOS}
	If $N$ is a positive integer and $f$ is an arithmetic function with period $N$ and $f(n) \in \{-1,1\}$ when $n= 1,2,\dots N-1$ and $f(n)=0 $ whenever $n\equiv 0 \bmod{~} N$ then $\sum \limits_{n \ge 1} \frac{f(n)}{n} \neq 0$. 
\end{Conj}
Following \cite{SIDDHI}, we describe the functions $f$ satisfying the condition of Conjecture \ref{CONJ:ERDOS} as Erd\"{o}s functions. We abbreviate the infinite sum as $ L(1,f) $ where $ L(s,f) $ denotes the Dirichlet series $ \sum_{n=1}^{\infty} f(n)n^{-s} $ for $ \Re{s}>1 $. The question of non-vanishing for special values of Dirichlet series not admitting an Euler product was first initiated by Chowla \cite{CHOWLA}. More precisely, the interest was to extend the result on the non-vanishing of $ L(1,\chi) $ for a non-principal Dirichlet character $ \chi $ to non-vanishing of $ L(1,f) $ for a rational valued arithmetic function $ f $ of period $ N $ whenever the sum converges. An important step towards this question was provided by Baker, Birch and Wirsing \cite{BAKER-BIRCH-WIRSING} where they proved the following result \footnote{This theorem is proved in a more general context with an extra condition on the field of values of $ f $}
:      
\begin{Th}\label{TH:BBW}
	Let $f$ be a non-zero rational valued arithmetic function of period $N$ satisfying $f(n) = 0 \text{ whenever } 1 < (n,N) < N$. Then $L(1,f) \neq 0$ whenever the sum converges.
\end{Th}
Here the condition $f(N)=0$ is relaxed; however apart from this, we have to assume that the value $f(n)$ is zero whenever $n$ is not co-prime to $N$. Note that this theorem resolves Conjecture \ref{CONJ:ERDOS} when $N$ is an odd prime. The proof of this theorem uses Baker's theorem on linear forms in logarithms of algebraic numbers (Theorem \ref{TH:BAKER}). Baker's theorem is required in their proof as we can write $L(1,f)$ as a linear form in logarithm of cyclotomic numbers by appealing to Plancherel's theorem. Using this idea along with Theorem \ref{TH:BBW}, Okada \cite{OKADA} gave a necessary and sufficient criterion for the vanishing of $L(1,f)$ and most of the attempts in understanding Erd\"{o}s Conjecture are based on applying this criterion (see \cite{RAM-TAPAS}, \cite{SARADHA-ERDOS1}, \cite{SARADHA-ERDOS2}, \cite{SARADHA-TIJDEMAN}). This theorem converts the question of vanishing of $ L(1,f) $ into a question in linear algebra where $ f(1),\dots, f(N) $ is a solution for a system of linear equations with rational coefficients depending on $ N  $. The only place where this condition is not used is in \cite{RAM-SARADHA}. Here the authors proved Conjecture \ref{CONJ:ERDOS} is true for $ N \equiv 3 \bmod{~} 4 $. Our goal is to give an alternative criterion for the vanishing of $ L(1,f) $ and study the consequences of such theorems to Erd\"{o}s conjecture and related questions.

We observe that a slightly stronger result than the non-vanishing of $L(1,f)$ is contained in all the proofs towards Erd\"{o}s conjecture and we encompass it in the following property : 

\settheoremtag{U}
\begin{property}\label{PROPERTY-U}
For a rational valued periodic function $f$ of period $N$, the value $L(1,f)$ does not belong to the vector space $\mathbb{Q}\langle \log d~:~ d \mid N \rangle$. 
\end{property}
Note that this Property implies the non vanishing of the infinite sum $L(1,f)$. Let $\omega(N)$ denote the number of distinct prime divisors of $N$. 
With this, we state the first theorem.
\begin{Th}\label{TH:ERD-INDUCTION}
    Let $N$ be squarefree positive number and assume that $f$ is an Erd\"{o}s function satisfying Property \ref{PROPERTY-U}. Then for any prime \[ p > 2^{\omega(N)} \prod\limits_{\substack{q \mid N\\q \text{ prime }}}(q^2 + 2q +2),\] 
    and for any Erd\"{o}s function $g$ of period $ pN $ such that $g(pn) = f(n)$ for all natural numbers $n$, we have that $g$ satisfies Property \ref{PROPERTY-U}.  
\end{Th} 
We would like to emphasize that the result is about functions of period $pN$ satisfying Property \ref{PROPERTY-U} under the assumption that the same is satisfied by function of period $N$. This theorem highlights an approach to prove Conjecture \ref{CONJ:ERDOS} via an descent argument on the number of prime divisors of $N$. We assume that the conjecture is true $\bmod{~}N$ and we would like to prove it is true $\bmod{~}pN$ whenever the prime $p$ does not divide $N$. With a mild assumption, Theorem \ref{TH:ERD-INDUCTION} shows that it is possible provided the prime $p$ is large enough. With this theorem and an additional calculation, we obtain the following corollary :
\begin{Cor}\label{COR:NEW-FAMILY-ERDOS}
	Let $N$ be a positive squarefree number congruent to $ 3 \bmod{~} 4$. Then for all primes 
	\[p > 2^{\omega(N)} \prod_{\substack{q \mid N\\q \text{ prime }}} ((q+1)^2+1),\] 
	the Erd\"{o}s conjecture is true for period $ pN $.
\end{Cor}
From this corollary, we obtain a new family of natural numbers $ N \equiv 1 \bmod{~} 4 $ for which Erd\"{o}s conjecture is true. This is done by choosing primes $p \equiv 3 \bmod {~} 4 $ in the above corollary. Such a result was not proved earlier due to the mysterious nature\footnote{The numbers $A(r,a)$ appearing in the theorem are not necessarily in the reduced form and can be zero.} of the rational numbers appearing in the Okada's theorem \cite[Theorem 10]{OKADA}. 

By a parity argument, it is possible to show that Erd\"{o}s Conjecture is vacuously true whenever $N$ is even, as there are no Erd\"{o}s functions $f$ for which the Dirichlet series $L(s,f)$ converges at $s=1$. Hence it is more instructive to replace the condition $f(N) = 0$ with $f(N) \in \{\pm 1\}$ for periodic functions $f$ of period $N$ provided $N$ is even. With this, we state our next result.
\begin{Th}\label{TH:4p}
Let $q$ be an odd prime and $f$ be an arithmetic function of period $4q$ taking values in $\{ \pm 1 \}$. Then $\sum\limits_{n = 1}^{\infty}\frac{f(n)}{n} \neq 0$ whenever it converges.  
\end{Th}
Using our method, we can also study similar questions whenever the period is a power of $2$. Our results rely on a new vanishing condition of $L(1,f)$ which we describe below : 
\subsection{A new vanishing condition for $L(1,f)$}
We recall that the rational valued arithmetic functions of period $ N $ which are supported on the set of co-prime residue classes modulo $ N $ are called Dirichlet type functions and we denote the set of such functions by $ \mathbf{F_D}(N) $. Furthermore, if we set $G=\text{Gal}(\mathbb{Q}(\zeta_N)/\mathbb{Q})$ where $\zeta_N = e^{2 \pi i / N}$, then we can define the action of $\sigma_a$ (given by $\sigma_a(\zeta_N) =\zeta_N^{a}$) on $\mathbf{F_D}(N)$ as follows $\sigma_a^{-1}f(n):= f(an)$ and  this action can be extended to the group ring $\mathbb{Q}[G]$ on $\mathbf{F_D}(N)$. 
\subsubsection{Definitions and Notations for Theorem \ref{TH:NEW-NECESSARY-SUFFCIENT-VANISHING-L1F}}
To describe the necessary and sufficient condition for the vanishing of $L(1,f)$, we require the following notations.
\begin{enumerate}
    \item For a prime $ p $, $ v_p $ denotes the normalised $ p $-adic valuation on the non-zero rational numbers.
    \item For every divisor $ d $ of $ N $, we define the function $ f_d(n): = f(dn)\chi_{0,N/d}(n)$ with $ \chi_{0,N / d} $ denoting the principal character $ \bmod{~} N / d $.
    \item We define $\mathbf{1}_p^{N}$ to the indicator function of integers $n$ satisfying $v_p(n) < v_p(N)$.   
\end{enumerate}
We now extend the definition of primitivity of characters to Dirichlet type functions : 
\par
\begin{Def}\label{DEF:PRIMITIVE}
	Let $f$ be a Dirichlet type arithmetic function of period $N$. We say $f$ is primitive $ \bmod{~} N$ if there exists a primitive Dirichlet character $\chi \bmod{~}N$ such that $ \sum_{a=1}^N f(a) \chi(a)\neq 0$.
\end{Def}
We say that the function $f$ is imprimitive  $ \bmod{~} N $ if it is not primitive $\bmod{~} N$. We also define ``projection'' functions. We can identify $\mathbf{F_D}(d)$ as a subspace of $\mathbf{F_D}(N)$ and we define the operator $ \Proj_d^N : \mathbf{F_D}(N) \to \mathbf{F_D}(d) $ as the projection under this identification. More precisely, we give the following definition.  
\begin{Def}\label{DEFN:INTRO-PROJ}
Let $f$ be a Dirichlet type function of period $N$ admitting the following character decomposition. 
\[ f = \sum_{ \chi \bmod{~} N} c_\chi \chi \] 
For a divisor $d$ of $N$, the function $\Proj_{d}^N(f)$ is a Dirichlet type function of period $d$ defined\footnote{We consider characters $\chi \bmod{~} d$ and the induced characters $\chi \bmod{~}N $ to be different and hence we multiply with the principal character $\chi_{0,N}$ to characters $\bmod{~} d$} as follows : 
\[ \Proj_d^N(f) = \sum_{ \chi \bmod{~} d} c_{\chi\chi_{0,N}}~\mathbf{\chi} \] 
\end{Def}
We now describe a necessary and sufficient condition for the vanishing of $L(1,f)$ in terms of imprimitivity of a certain Dirichlet type function $\bmod{~} d$, for every divisor $d$ of $N$.
\begin{Th}\label{TH:NEW-NECESSARY-SUFFCIENT-VANISHING-L1F}
Let $f$ be a rational valued arithmetic function of period $N$ such that the Dirichlet series $L(s,f)$ converges at $s=1$. Then $L(1,f) = 0 $ if and only if the following conditions are satisfied : 
\newline
For \textbf{every divisor} $D$ of $N$, the arithmetic function :
\begin{equation}\label{EQN:INIT-COND-DIVISORS}
	\sum_{d \mid D} \frac{1}{d} \prod_{\substack{p \mid N/d\\p \nmid N/D}}\big(1-\frac{\Frob{p}^{-1}}{p}\big) \Proj_{N/D}^{N/d}(f_{d}) \text{ is } \textbf{imprimitive} \bmod{~}N/D. \tag{$V_D'$}
\end{equation}
For \textbf{every prime divisor} $p$ of $N$, we have : 
\begin{equation}\label{EQN:COND-PRIMES}
	\sum_{d \mid N} (\sum_{a=1}^{N/d}f_d(a)) \bigg(v_p(d) - \frac{\mathbf{1}_p^N(d)}{p-1}\bigg) = 0 \tag{$ V_p^{\text{pr}} $}
\end{equation}
\end{Th}
In \eqref{EQN:INIT-COND-DIVISORS}, $d$ varies over all the divisors of $D$ and $p$ varies over certain prime divisors of $N/d$. By convention, if $ \frac{N}{D} \equiv 2 \bmod{~} 4 $, the sum \eqref{EQN:INIT-COND-DIVISORS} is imprimitive $\bmod{~}N/D$, as there are no primitive characters of conductor congruent to $ 2 \bmod{~} 4 $.

From our definition, we have $ \Proj_N^N(f_1) = f_1 $ and when we set $D=1$ at \eqref{EQN:INIT-COND-DIVISORS}, we arrive at the following corollary. 
\begin{Cor}\label{COR:N-prim-nonvanishing}
Let $f$ be a rational valued periodic function of period $N$. If there exists a \textit{primitive} character $\chi \bmod{~}N$ such that $\sum\limits_{a=1}^Nf(a) \chi(a) \neq 0$, then $L(1,f) \neq 0$ whenever the infinite sum converges.
\end{Cor}
When $f$ is an even function (that is $f(-n)= f(n)$ for all integers $n$), the above corollary can be obtained by the $\bQ$-linear independence of $L(1,\chi)$ as $\chi$ varies over the non-trivial even characters $\bmod{~} N$ (See \cite[Theorem 3]{BAKER-BIRCH-WIRSING}) using the method we describe. However, for the case when $f$ is not even, the result is completely new. Moreover, we are not using the $\bQ$-linear independence of $L(1,\chi)$ in our proof.

According to Definition \ref{DEF:PRIMITIVE}, to understand the imprimitivity of a given Dirichlet type function $f$ of period $N$, we have to compute $\sum_{a=1}^Nf(a)\chi(a)$ and the value will lie in $\mathbb{Q}(\chi)$. It is desirable to do the computations in the smallest field containing the values of $f$. Given a rational valued Dirichlet type function $f$ of period $N$, we associate an element $\mathcal{P}(f) \in \mathbb{Q}[G]$ as follows : 
\begin{equation}\label{EQN:POLY-ASSOCIATE}
	\mathcal{P}(f) := \sum_{(a,N)=1} f(a) \sigma_a.  
\end{equation}
In the above, the sum runs over a finite set of complete representatives in $ (\mathbb{Z}/N\mathbb{Z})^* $. Henceforth we identify $f$ with $\mathcal{P}(f)$. We prove a Proposition which helps us to identify functions $f$ which are imprimitive $\bmod{~} N$, and also describe $\mathcal{P}(\Proj(f))$ in terms of $\mathcal{P}(f)$. With the help of Proposition \ref{PROP:POLY-NPRIM-ANNHILATE}, for each divisor $D$ of $N$, we can re-write \eqref{EQN:INIT-COND-DIVISORS} as follows :
\begin{equation}\label{EQN:COND-DIVISORS:INTERMS-GROUPRING}
	\prod_{p \mid N / D} (1-\tau_p^{(N/D)}) \big(  \sum_{d \mid D} \frac{\varphi(N/D)}{d\varphi(N/d)} \sideset{}{'}\prod_{p}(1-\frac{\Frob{p}^{-1}}{p}) \mathcal{P}(f_{d})\big\vert_{\mathbb{Q} (\zeta_{N / D})} \big) = 0. \tag{$V_D$ } 
\end{equation}
Here $\tau_p^{(N/D)}$ is a generator of the Galois group $\Gal{\mathbb{Q}(\zeta_{N/D})/\mathbb{Q}(\zeta_{N/Dp})}$ and we are restricting  $\mathcal{P}(f_d)$ to $\mathbb{Q}[\Gal{\mathbb{Q}(\zeta_{N/D})/\mathbb{Q}}]$ by restricting each automorphism accordingly. Also $\varphi$ denotes the Euler phi function. We use this equation along with \eqref{EQN:COND-PRIMES} to study the non-vanishing of $L(1,f)$ for periodic functions $f$ and prove Theorem \ref{TH:4p} by applying this condition. 
\par

\par
The techniques used in the main theorem can be applied for understanding the structure of $ L(s,f) $ having finitely many zeroes at positive integers. The non-vanishing for special values of Dirichlet series $L(k,f)$ for rational valued periodic arithmetic functions and for integers $k>1$ was first initiated by P. Chowla and S. Chowla in \cite{CHOWLA-CHOWLA}. Later on, Milnor \cite{MILNOR} using Kubert's relation generalised the conjecture when the period is not a prime.  
\begin{Conj}[Milnor]\label{CONJ:CHOW-MIL} For positive integers $q,k$ greater than one, the values $\zeta(k,a/q)$ with $a$ satisfying $1 \le a \le q$ and $(a,q)=1$ are linearly independent over $\mathbb{Q}$.  
\end{Conj}
The above conjecture is often referred to as Chowla Milnor conjecture (see \cite{GUN-MURTY-RATH} and subsequent works). We prove the following theorem assuming this conjecture. 
\begin{Th}\label{TH:CHARACTERISATION-2PRIMES}
	Let $f$ be a rational valued function of period $ pq $ where $ p \text{ and } q $ are distinct primes. Suppose we have distinct positive integers $ k,l $ greater than one such that such that $ L(k,f) = L(l,f) = 0  $. Then 
	\[ L(s,f) = \bigg(c_1 \big(1-\frac{p^k}{p^s}\big)  \big(1-\frac{q^l}{q^s}\big) +  c_2\big(1-\frac{p^l}{p^s}\big)  \big(1-\frac{q^k}{q^s}\big) \bigg) \zeta(s),    \]
	for some rational numbers $ c_1 \text{ and } c_2 $, conditional to Conjecture \ref{CONJ:CHOW-MIL}. 
\end{Th}
From this theorem, one can conclude that the zeroes of $ L(s,f) $ for $\Re{s}>1$ are determined by the vanishing of a Dirichlet polynomial. We can generalise the statement to functions $ f $ having squarefree period $ N $, but at the moment, it is computationally taxing.  
\subsection{New ingredients and comparison with earlier methods}
In this section, we list the new ingredients that were not present in the earlier works and the necessary tools required.
\begin{enumerate}
    \item \textbf{Direct sum decomposition :} The idea behind the proof of Theorem \ref{TH:NEW-NECESSARY-SUFFCIENT-VANISHING-L1F} relies on a direct sum decomposition of the vector space $\mathbf{F}_D(N)$ consisting of Dirichlet type functions. The proof of Okada \cite{OKADA} relies on a rearrangement of certain values and construct functions $g$ appropriately such that $L(1,g) = 0$. We use a direct sum decomposition based on conductor of characters. We require to compute the value of each summand. The author along with S. Pathak considered a character decomposition in \cite[Section 4]{ABHISHEK-SIDDHI} but this was again used for rearrangement and matching the values $L(s,f)$ and $L(s,g)$ for different periods.  In the proofs of Theorem \ref{TH:NEW-NECESSARY-SUFFCIENT-VANISHING-L1F} as well as \cite[Theorem 10]{OKADA} re-arrangements and the use of Theorem \ref{TH:BBW} are unavoidable.
    \item \textbf{Group rings :} The idea of tracing back the vanishing of $L(1,f)$ to vanishing of certain elements in a group ring can be found in \cite{BHARADWAJ-THESIS} when $f$ is an odd function of period $N$. In this case, the mapping to group rings is possible as $L(1,f)$ is an algebraic multiple of $\pi$, and there only one relation among the Galois conjugates of the cotangents 
    $\{\sigma_a(\cot{2\pi/N})~:~ \sigma_a \in \text{Gal}(\mathbb{Q}(\zeta_N)/\mathbb{Q}) \}$, namely $\sigma_{-1}(\cot{2\pi/N}) = -\cot{2\pi/N}$ (see \cite{OKADA-COTANGENTS}, \cite{GIRSTMAIR}). For periodic functions $f$ without any constraints, we obtain a contribution from linear forms in logarithms of positive algebraic numbers. Our formalism avoids these expressions of linear forms in logarithms. There are some similarities in our approach and that of \cite{LEOPOLDT}, \cite{GIRSTMAIR-CHARCOODS}. However, unlike \cite{GIRSTMAIR-CHARCOODS}, we are not computing character coordinates. One can say that we are providing a representation theoretic interpretation for the vanishing of $L(1,f)$. We indirectly show that it is possible to arrive at a vanishing condition of $L(1,f)$ using Dirichlet characters $\bmod{~}N$.  
    \item \textbf{Vanishing when the period is even : } When $N$ is even, we can construct examples of functions $f$ with values in $\{\pm 1\}$ of period $N$ for which Property \ref{PROPERTY-U} does not hold. This section is partially motivated by the works  \cite{MURTY-SIDDHI}, \cite{TIJDEMAN} and \cite{PILERUD}. Tengely had given a counterexample for the vanishing of $L(1,f)$ when $f$ is of period 36 taking values in $\{\pm 1\}$. The proof given in \cite{PILERUD} uses the relations between the digamma values whereas we observe that the vanishing is due to the nature of zeroes of certain Dirichlet polynomials at $s=1$. Also, our proof of Theorem \ref{TH:4p} provides us an understanding on how to apply Theorem \ref{TH:NEW-NECESSARY-SUFFCIENT-VANISHING-L1F} to prove questions similar to that of Erd\"{o}s conjecture. In some sense, $N=4p$ is a non-trivial ``toy-case'' not covered in the literature and we apply our results here.
    \item \textbf{Solving a system of linear equations : } The earlier progress towards Erd\"{o}s Conjecture involved solving \textit{only} one equation : either it is finding the algebraic component $\alpha$ of $\alpha \log u$ where $u$ is a unit in $\Z[\zeta_N]$ and prove $\alpha$ is non-zero by arithmetic considerations (like \cite{RAM-SARADHA}) or use one of the linear equations given in Okada's condition to obtain a bound on the coefficients by taking absolute values or minor considerations (like \cite{SARADHA-ERDOS1}, \cite{SARADHA-ERDOS2}, \cite{RAM-TAPAS}). To obtain better results, we use our criteria to solve a system of equations. The arithmetic consideration we require is the impact of imprimitivity of $f \in \mathbf{F_D}(N)$ on the primitive component of $\Proj_{N/p}^N(f)$ (See Proposition \ref{PROP:IMPRIMITIVE-INDUCTION}). 
    
\end{enumerate}
\subsection{Organisation of the paper}
We begin with some preliminaries and consequences of Baker's theorem in Section \ref{SEC:PRELIM}. We prove Theorem \ref{TH:NEW-NECESSARY-SUFFCIENT-VANISHING-L1F} in Section \ref{SEC:PROJD-TH-VANISHING}. In Section \ref{SEC:GROUP-RINGS}, we associate our function $f$ with an element in the group ring $\mathbb{Q}[\Gal{\mathbb{Q}(\zeta_N)/\mathbb{Q}}]$ and prove some properties associated to it. The properties we focus are namely the representation of $\Proj_{d}^N$ in terms of group ring elements, necessary and sufficient condition on imprimitivity of a function $f \in \mathbf{F_D}(N)$ and the impact of imprimitivity of $f$ on $\Proj_{N/p}^N(f)$ for a prime $p$ dividing $N$. In Section \ref{SEC:TH-4P}, we prove Theorem \ref{TH:4p} and give a simple proof of a counterexample given by Tengely on vanishing of $L(1,f)$ when $f$ takes values $\{ \pm 1 \}$. In Section \ref{SEC:ERD-INDUCTION} we prove Theorem \ref{TH:ERD-INDUCTION} and in Section \ref{SEC:CHOWLA-MILNOR} we proceed to the proof of Theorem \ref{TH:CHARACTERISATION-2PRIMES}. We end by pointing out some advantages and disadvantages of our approach along with some future directions in Section \ref{SEC:CONCLUDING-REMARKS}.

\subsection{Notations} Throughout, $p$ and $q$ will denote primes, we use $d,e,D$ to denote divisors of natural numbers $N$. The letters, $\alpha, \beta, \sigma, \tau$ will be used to denote automorphisms or elements of group rings, and the letters $\chi,\Psi$ along with their subscripts if required will be used to denote characters. The letters $f,g,h,F,G,H$ will be used for functions. We do not mention these explicitly in every Lemma/Proposition/Theorem. Sometimes we use $\widetilde{f}$ to denote explicit dependence on the function $f$.  
\section{Basic definitions and Remarks}\label{SEC:PRELIM}
\subsection{ Preliminaries }We recall that $ \mathbf{F_D}(N)$ denotes the $ \mathbb{Q} $-vector space of Dirichlet type functions of period $ N $. This also admits a natural inner product structure given by $ \langle f,g \rangle = \varphi(N)^{-1}\sum_{a=1}^N f(a)\overline{g}(a) $. Here $\overline{g}$ denotes the complex conjugate of $g$ and this inner product structure is defined for arithmetic functions taking complex values.

Given a function $ f $ of period $ N $, the Dirichlet series $ L(s,f) $ admits a possible simple pole at $ s=1 $ and to make sense of $ L(1,f) $ we require the condition $ \sum_{a=1}^N f(a) = 0 $.  We now state the following definition: 
\begin{Def}\label{DEFN:PUREDPRIM}
A Dirichlet type arithmetic function $g$ of period $N$ is said to be \textit{purely} $d$-primitive if $\langle g,\chi \rangle = 0$ for all characters $\chi$ which are not induced from primitive characters $\chi' \bmod{~} d$.   
\end{Def}
The $ \mathbb{Q} $-vector space consisting of  \textit{purely} $d$-primitive rational valued Dirichlet type functions of period $ N $ is denoted by $\mathbf{F_D}^{pr}(N;d)$.

We use the notation $\text{cond}({\chi})$ to denote the conductor of a Dirichlet character $\chi$. Given a rational valued Dirichlet type function $g$ of period $N$, we have 
\[ g = \sum\limits_{\chi \bmod{} N} \langle g,\chi \rangle \chi = \sum\limits_{d \mid N} \sum\limits_{\substack{\chi\bmod{} N\\ \text{cond}(\chi) = d}}\langle g,\chi \rangle \chi = \sum_{d \mid N} \mathcal{D}_N(d,g), \]
where we define a linear map $\mathcal{D}_N(d,$\textunderscore $) : \mathbf{F_D}(N) \to \mathbf{F_D}(N)$ as follows :
\begin{equation}\label{EQN:D_N-DEF}
  \mathcal{D}_N(d,g):= \sum\limits_{\substack{\chi\bmod{} N\\ \text{cond}(\chi) = d}} \langle g,\chi \rangle \chi.  
\end{equation}

It can be verified that the image of $ \mathcal{D}_N(d,g) $ is rational valued whenever $ g $ is rational valued. This is true since for $\sigma \in \text{Gal}(\bQ/\mathbb{Q})$, we have  $\sigma(\langle g,\chi \rangle ) = \langle g,\chi^\sigma \rangle $ (here $\chi^\sigma(n):=\sigma(\chi(n))$). Since we are summing over all the characters $\chi$ of conductor $d$, we note that the overall sum $\mathcal{D}_N(d,g)$ is invariant under the Galois action and hence is rational valued. 

By the orthogonality of characters, we have a direct sum decomposition of $ \mathbf{F_D}(N) $ namely, 
\begin{equation}\label{EQN:DIRECT-DECOMP-DIRITYPE-FUNCTION}
  \mathbf{F_D}(N) = \bigoplus\limits_{d\mid N} \mathbf{F_D}^{pr}(N;d),  
\end{equation}
via the maps $\mathcal{D}_N(d,\_)$. Also, note that the map $\mathcal{D}_N(d,\_)$ is well behaved under the action of $\mathbb{Q}[G]$. 
\par
For integers $ e \mid d \mid N $ and a function $ f \in \mathbf{F_D}(N) $, we denote the arithmetic function ${\mathcal{D}_d(e,N) \in \mathbf{F_D}(d)}$ by the following : 
\[ \mathcal{D}_d(e,f) = \sum_{\substack{\chi \bmod{} d \\ \text{cond} \chi = e }} \langle f,\chi \rangle \chi. \]
In the above expression, by $ \langle f,\chi \rangle $, we are taking the inner product with the character $\chi'$ of period $N$ induced from the character $\chi$ of period $d$. (Note the change in the conductor of the character here as compared to \eqref{EQN:D_N-DEF}).
\begin{Rem}\label{REM:DEF:IMPRIM:REINTEPRET}
	With the definition of linear map $ \mathcal{D}_N(N, -) $, we can say that a Dirichlet type function $ f $ of period $ N $ is imprimitive mod $ N $ if and only if $ \mathcal{D}_N(N,f) = 0 $.
\end{Rem}
Let $ d,N $ be two positive integers with $ d \mid N $. Given $ f \in \mathbf{F_D}(d) $, we have an arithmetic function ${f \chi_{0,N} \in \mathbf{F_D}(N) }$. This is an injective map, as given any $ a $ co-prime to $ d $, we can find a number $ b $ co-prime to $ N$ with $ b \equiv a \bmod{~} d $. This helps us to define the preimage of a function $ h \in \mathbf{F_D}(d)\chi_{0,N} $ in $ \mathbf{F_D}(d) $ denoted as follows : 
\[ \mathbf{F_D}(d)\chi_{0,N} \to \mathbf{F_D}(d) \qquad h \to h\vert_{d}. \]
Note that the values of the functions $ h\vert_{d} $ and $h$ are same except at certain integers. Hence we can identify $ \mathbf{F_D}(d)$ as a subspace of $ \mathbf{F_D}(N) $.  We now give an equivalent definition of $\Proj^N_d$ mentioned in Section \ref{SEC:INTRO}.
\begin{Def}\label{DEF:PROJ}
	Let $ d,N $ be two integers such that $ d \mid N $ and $ f \in \mathbf{F_D}(N)$. Let $ h $ be the projection of $ f $ onto $ \mathbf{F_D}(d) \chi_{0,N}$. The function $ \Proj_d^N(f) $ is defined as the arithmetic function $ h\vert_d$ in $ \mathbf{F_D}(d) $.
\end{Def}
\begin{Rem}\label{REM:PROJ}
	The operator $ \Proj_{d}^N : \mathbf{F_D}(N) \to \mathbf{F_D}(d)$ can be written as
	\[ \Proj_d^N(f) = \sum_{e \mid d} \mathcal{D}_d(e,f), \]
	where $ f $ admits the decomposition $ f = \sum_{ e \mid N} \mathcal{D}_N(e,f) $. This is true as $ \mathcal{D}_d(e,f) = \mathcal{D}_N(e,f)\vert_d $. 
\end{Rem}

\subsection{Some linear independence Results using Baker's Theorem}
We start with the statement of Baker's theorem (A quantitative version on the lower bound of this theorem is given in \cite{BAKER}) : 
\begin{Th}\label{TH:BAKER}
Let $a_1,\dots, a_n$ be multiplicative independent numbers. Then their logarithms 

{$1,\log a_1,\dots,\log a_n$} are linearly independent over the field of algebraic numbers $\overline{\mathbb{Q}}$.
\end{Th}
We require the following consequence of the above theorem (See \cite[Lemma 3.6]{ABHISHEK-SIDDHI}). 
\begin{Lemma}\label{LEM:ALG-UNITS-INDEPENDENT-PRIMES-LOG}
Let $\mathbb{F}$ be a number field. Suppose that $u_1,\dots, u_n \in \mathcal{O}_{\mathbb{F}}^*$ and let $S$ be a finite set of rational primes. Then
\[ \overline{\mathbb{Q}}\langle \log u_i ~|~ 1 \le i \le n \rangle \cap \overline{\mathbb{Q}}\langle \log p~|~ p \in S\rangle = \{0\}. \]
\end{Lemma}
We also require the following Proposition (This result is contained in \cite[Theorem 3.7]{ABHISHEK-SIDDHI}).  
\begin{Prop}\label{PROP:DIRI-TYPE-FUNCTION_L1F}
Let $f$ be a rational valued Dirichlet type function of period $N$ satisfying $\sum_{a=1}^N f(a) = 0$. Then $L(1,f)$ is a linear combination of logarithm of algebraic units in $\Z[\zeta_N]$. Consequently by Lemma \ref{LEM:ALG-UNITS-INDEPENDENT-PRIMES-LOG}, $f$ satisfies Property \ref{PROPERTY-U}.
\end{Prop}
We finish this section with two remarks on application of the above Proposition.

\begin{Rem}\label{REM:DPRIMITIVE-PROPERTY-U}
For a non unitary positive divisor $D$ of $N$ and $f \in \mathbf{F_D^{pr}}(N)$, from the above Proposition we note that $f$ satisfies Property \ref{PROPERTY-U}.  
\end{Rem}
\begin{Rem}\label{REM:OKADA-PROPERTY-U}
In \cite[Theorem 1]{OKADA-JLMS}, given a function $f$, Okada had constructed an explicit Dirichlet type function $g$ of period $N$ depending on $f$ such that $L(1,f) = L(1,g)$ (See also \cite[Lemma 1]{OKADA-JLMS}) . If $L(1,f) = 0$, then by Theorem \ref{TH:BBW}, we obtain $g \equiv 0$. Finding the value for each $g$ along with the contribution from the residues, we get $\varphi(N) + \omega(N)$ linear equations. As mentioned in \cite[Lemma 4]{OKADA-JLMS}, this corresponds the values of $g$ corresponds to the equations mentioned in \cite[Theorem 10]{OKADA}. However since $g$ is a Dirichlet type function of period $N$, $g$ satisfies Property \ref{PROPERTY-U} by Proposition \ref{PROP:DIRI-TYPE-FUNCTION_L1F}. 
\end{Rem}

\section{Proof of Theorem \ref{TH:NEW-NECESSARY-SUFFCIENT-VANISHING-L1F}}\label{SEC:PROJD-TH-VANISHING}
We first mention an equivalent criterion for $(V_D')$ for every divisor $D$ of $N$ in terms of the function $\mathcal{D}_{N/D}(N/d,f_d)$ for every $d \mid D$ ( See equation \ref{EQN:REARRANGED-PRIM-PART} ). 
\begin{equation}\label{EQN:INIT-COND-DIVISORS-EQUIVALENTFORM}
	\sum_{d \mid D} \frac{1}{d} \prod_{\substack{p \mid N/d\\p \nmid N/D}}\big(1-\frac{\Frob{p}^{-1}}{p}\big) \mathcal{D}_{N/D}(N/d,f_{d}) = 0 \end{equation}
The introduction of $\Proj_d$ is new in this topic and according to Proposition \ref{PROP:POLY-NPRIM-ANNHILATE} it behaves like a relative trace operator. This imitates the situation in the case for the relative traces of elements for cyclotomic fields. When we do a direct sum decomposition of a Dirichlet type function $f$ of period $N$ as mentioned in \eqref{EQN:DIRECT-DECOMP-DIRITYPE-FUNCTION} we can obtain the above equation. However, the values of the components $\mathcal{D}_N(d,g)$ for a divisor $d$ of $N$ are hard to compute and require an inclusion-exclusion principle. The disadvantage of finding the direct sum component is that the denominators become large. This is the main reason for introducing Definitions \ref{DEF:PRIMITIVE}, \ref{DEF:PROJ} and proving Lemma \ref{LEM:PRIMPART-ALMOST-SAME-PROJPART}. In some sense, this lemma is about the relation between functions $g \in \mathbf{F_D}(N)$ and $\Proj_d^N(g)$. We investigate the special values of their ``$d-$primitive'' parts for a divisor $d$ of $N$.

We recall that if $g \in \mathbf{F_D}(m)$ and if $\sigma_a \in \Gal{\mathbb{Q}(\zeta_m)/\mathbb{Q}}$ is given by $\sigma_a(\zeta_m) = \zeta_m^a$, then we define the action as follows : $\sigma_a^{-1} g(n) = g(an)$. We now consider the character decomposition of $g$ and the action of the Galois group. We have the following :  
\begin{equation}\label{EQN:GALOISACTION}
g = \sum_{\chi \bmod{} m} c_\chi \chi \implies \sigma_a.g = \sum_{\chi \bmod{} m} \chi(a)^{-1}c_\chi \chi. 
\end{equation}  
\begin{Lemma}\label{LEM:PRIMPART-ALMOST-SAME-PROJPART}
Let $d \neq 1 $ be divisor of $N$ and let $g \in \mathbf{F_D}(N)$. We have 
\[L(1,\mathcal{D}_N(d,g)) = L(1,\mathcal{D}_d(d,\beta \Proj_d^N(g))),\]
where $\beta = \prod_{\substack{p \mid N\\p \nmid d}}(1-\frac{\sigma_p^{-1}}{p})$. 
\end{Lemma}
\begin{proof}
Let $g = \sum\limits_{\chi \bmod{} N} \langle g,\chi \rangle \chi$. We have
\begin{flalign}\label{EQN:L1F-CHARACTERBASIS-MOBIUSWAY}
L(1,\mathcal{D}_N(d,g)) & = \sum\limits_{\substack{\chi\bmod{} N\\ \text{cond}(\chi) = d}} \langle g,\chi \rangle L(1,\chi) = \sum\limits_{\substack{\chi\bmod{} d\\ \text{cond}(\chi) = d}}\langle g,\chi\rangle \prod_{p \mid N}\big(1-\frac{\chi(p)}{p}\big) L(1,\chi) \nonumber. \\
&= \sum_{e \mid N} \frac{\mu(e)}{e} \sum\limits_{\substack{\chi\bmod{} d\\ \text{cond}(\chi) = d}}\langle g,\chi \rangle \chi(e)L(1,\chi) \nonumber. \\
&=\sum_{e \mid N} \frac{\mu(e)}{e} L(1,\mathcal{D}_d(d,\sigma_{e}^{-1}.\Proj_d^N(g))) \text{ by } \eqref{EQN:GALOISACTION} \nonumber.
\end{flalign}
Noting that $\sum_{e \mid N} \frac{\mu(e)\sigma_{e}^{-1}}{e} = \beta$, we obtain the result. 
\end{proof}
Let $W_d^{pr}(N):= \{ L(1,f) ~\big\vert ~f \in \mathbf{F_D}^{pr}(N;d) \text{ and } \sum_{a=1}^N f(a) = 0 \}$. We note that it is a $\mathbb{Q}$-vector space. 
\begin{Cor}\label{COR:SAME-DPRIM-VAL}
$W_d^{pr}(N)$ is independent of $N$. More precisely, $W_d^{pr}(N) = W_d^{pr}(d) $. 
\end{Cor}
\begin{proof}
	Let $g \in \mathbf{F_D}^{pr}(N;d)$. By the above lemma, we have $L(1,g) = L(1,\beta \Proj_d^N(g))$ and therefore, $W_d^{pr}(N) \subseteq W_d^{pr}(d)$ as $ \mathbf{F_D}(d) $ is a $ \mathbb{Q}[\Gal{\mathbb{Q}(\zeta_d) / \mathbb{Q}}] $-module. To prove the reverse containment note that $ D_d(d,g) = g $ for all $ g \in \mathbf{F_D}^{pr}(d;d)$. Now, given a function $ g \in \mathbf{F_D}(d) $, there exists an arithmetic function $ h \in \mathbf{F_D}(N) $ such that $ h=\Proj_d^N(g) $. We now replace $ g $ by $ \beta^{-1}g $ in the above lemma to get the result. 
\end{proof}
The proof of Theorem \ref{TH:NEW-NECESSARY-SUFFCIENT-VANISHING-L1F} also requires the definitions of $ f_d,\Proj_d \text{ and } \mathcal{D}_N(-,f) $. 
We recall that the arithmetic function $f_d$ is defined by $f_d(n):=f(dn)\chi_{0,N/d}(n)$ and is of period $N/d$. Since $\sum_{a=1}^{N/d} f_d(a)$ need not be equal to 0, we cannot directly evaluate $L(s,f_d)$ at $s=1$. We instead evaluate the derivative of $(s-1)d^{-s}L(s,f_d)$ at $s=1$. 
\begin{proof}[Proof of Theorem \ref{TH:NEW-NECESSARY-SUFFCIENT-VANISHING-L1F}]
We can write $L(s,f) = \sum_{e \mid N}e^{-s} L(s,f_e)$. We now note that 
\begin{equation}\label{EQN:L1FasDERIVATIVE}
   L(1,f) = \frac{d}{ds}(s-1)L(s,f)\bigg\vert_{s=1} = \sum_{e \mid N} \frac{d}{ds} \big(\frac{1}{e^s}(s-1)L(s,f_e)\big)\bigg\vert_{s=1}. 
\end{equation}
The first equality is true by noting that in a neighborhood of $s=1$ we have $L(s,f) = L(1,f) + O(s-1)$. We now consider the character decomposition of $ f_d $. Note that the contribution for residue of $L(s,f_d)$ at $s=1$ is from the trivial character $\bmod{~}{N/d}$ as $L(s,\chi)$ for non-principal characters $\chi$ are entire. We write 
\begin{flalign*}
	\frac{d}{ds} \big(\frac{1}{e^s}(s-1)L(s,f_e))\big)\bigg\vert_{s=1}&= \frac{1}{e}\sum_{\substack{\chi \bmod{} N/d\\ \chi \neq 1}}\langle f_e,\chi\rangle L(1,\chi) \\
	&+ \langle f_e,\chi_{0,N/e}\rangle \frac{d}{ds}\big(\frac{1}{e^s}(s-1)L(s,\chi_{0,N/e})\big)\bigg\vert_{s=1}. \\
&=:S_e + \langle f_e,\chi_{0,N/e}\rangle R_e.
\end{flalign*}
Since $L(1,f) = 0$, we obtain 
\begin{equation}\label{EQN:S_d-R_d-BEFORE-SPLIT}
\sum_{d \mid N} (S_d + \langle f_d,\chi_{0,N/d}\rangle R_d) = 0.
\end{equation}
\textbf{Total contribution from $R_d$ : } We recall that $\zeta(s) = (s-1)^{-1} + \gamma + O(s-1)$. For a principal character $\chi_{0,N/d}$ the Dirichlet $L$ function associated to it is given by 
\[ L(s,\chi_{0,N/d}) = \prod_{p \mid N/d} \big(1-\frac{1}{p^s}\big)\zeta(s), \]
where $p$ runs over the prime divisors of $N/d$. Now we can evaluate $R_d$ and we obtain
\begin{equation*}
R_d= \frac{\varphi(N/d)}{N} \big(-\log d + \sum_{p \mid N/d} \frac{\log p}{p-1}  + \gamma \big).
\end{equation*}
Therefore, 
\begin{equation}\label{EQN:R_d-CONTRIBUTION}
  \sum_{d \mid N} \langle f_d,\chi_{0,N/d}\rangle R_d = \frac{-1}{N}(\sum_{d \mid N} {\sum_{a=1}^{N/d}}f_d(a) (\log d - \sum_{ p \mid N/d}\frac{\log p}{p-1})).   
\end{equation}
In the above, we use $\sum_{a=1}^N f(a) = 0$ to conclude that the coefficient of $\gamma$ is $0$.

Now, by the observation mentioned in the proof of Proposition \ref{PROP:DIRI-TYPE-FUNCTION_L1F}, we conclude that $S_d$ is a $\bQ$-linear form of logarithm of units in $\Z[\zeta_{N}]$. Therefore we can apply Lemma \ref{LEM:ALG-UNITS-INDEPENDENT-PRIMES-LOG} to \eqref{EQN:S_d-R_d-BEFORE-SPLIT} and we have 
\[\sum_{d \mid N} S_d = 0, \qquad \qquad \sum_{d \mid N} \langle f_d,\chi_{0,N/d} \rangle R_d = 0,\]
where in the first sum $d$ varies over the divisors of $N$ excluding $1$. Decomposing $\log(d) = \sum_{p \mid d} v_p(d) \log p$ and applying Theorem \ref{TH:BAKER} to \eqref{EQN:R_d-CONTRIBUTION}, we arrive at \eqref{EQN:COND-PRIMES}.

\textbf{Total contribution from $S_d$ : } We compute the contribution from the non-trivial characters $\bmod{~} N/d$ for each divisor $d$ of $N$. Writing in terms of special values of purely primitive functions, we have 
\begin{flalign*}
	S_d = \frac{1}{d} \sideset{}{'}\sum_{e \mid N/d}L(1,\D_{N/d}(e,f_d)),
\end{flalign*}
where $\sideset{}{'}\sum$ indicates the summation over the divisors of $ N/d $ excluding $1$. Therefore, substituting the above expression for $S_d$ in $\sum_{d \mid N}S_d = 0$ and interchanging the summation, we have : 
\begin{equation*}
    \sideset{}{'}\sum_{e \mid N} \sum_{d \mid N/e} \frac{1}{d} L(1,\D_{N/d}(e,f_d)) =0.
\end{equation*}
Setting $\beta_{N/d,e} = \prod\limits_{\substack{p \mid N/d\\p \nmid e}}(1-\sigma_p^{-1}p^{-1})$ and applying Lemma \ref{LEM:PRIMPART-ALMOST-SAME-PROJPART}, we have 
\begin{equation}\label{EQN:REARRANGED-PRIM-PART}
	\sideset{}{'}\sum_{e \mid N} \sum_{d \mid N/e} \frac{1}{d} L(1,\mathcal{D}_{e}(e,\beta_{N/d,e}.\Proj_{e}^{N/d}(f_d))) =0.
\end{equation}
We note that the inner sum (denoted by $\alpha_{e}$) is an element of $W_{e}^{pr}$ and therefore, we have 
\[ \sideset{}{'}\sum_{d \mid N} \alpha_d = 0, \text{ with } \alpha_d \in W_d^{pr}(d).\]

\textbf{Multiple Baker-Birch Application : } We recall that 
\begin{equation}\label{EQN:KEY-POINT}
	\text{ The map } \text{Ev}_1 : \mathbf{F_D}(N)^0 \to \mathbb{C} \text{ given by } f \mapsto L(1,f) \text{ is injective. } \tag{$ \textbf{Ev}_1 $}
\end{equation}
Therefore, by \eqref{EQN:DIRECT-DECOMP-DIRITYPE-FUNCTION} and Corollary \ref{COR:SAME-DPRIM-VAL} we note that $\sum_{d \mid N}W_d^{pr}(d)$ is a direct sum. Hence we conclude that $\alpha_d = 0$ for all divisors $d \mid N$. With the help of Theorem \ref{TH:BBW}, this translates to, 
\begin{equation*}
    \D_{e}(e,\sum_{d \mid N/e} \frac{1}{d} \beta_{N/d,e}.\Proj_{e}^{N/d}(f_d)) =0
\end{equation*}
thereby proving \eqref{EQN:INIT-COND-DIVISORS} for all divisors $e \neq 1$ of $N$ by Remark \ref{REM:DEF:IMPRIM:REINTEPRET}. It remains to verify the condition for $D=N$ and it can be shown that this is the same as saying $\sum_{a=1}^Nf(a) = 0$. We prove the case $D=N$ by computing $\Proj{d}_1$ where $d$ runs over the divisors of $N$. By Remark \ref{REM:PROJ} and noting that the Galois action is trivial on the rational numbers, we sum all the terms mentioned in $(V_N')$ together, we find that the overall sum is $N^{-1}(\sum_{a=1}^Nf(a))$ which is zero.
\end{proof}
We observe a curious property about $(V_p^{pr})$ for a prime divisor $p$ of $N$ under certain conditions. 

\textbf{Additional Terminology :} Let $f$ be an arithmetic function of period $N$ with $\sum_{a=1}^N f(a) = 0$ and $u$ be an algebraic number which is not 0 or 1. We have seen that $L(1,f)$ is a $\overline{\mathbb{Q}}$ linear form in logarithm of algebraic numbers. We write $L(1,f) = b \log u + C$ where $C $ lies in the orthogonal complement of $ \mathbb{Q}\langle \log u \rangle$. By \textbf{coefficient of} $\log u$ in the expansion of $L(1,f)$, we mean the value $b$.

\begin{Cor}\label{COR:LOGP0-CONSEQUENCE}
Let $N$ be a natural number greater than one. Suppose there exists a prime divisor $p$ of $N$ such that $v_p(N) = 1$. Let $f$ be a rational valued periodic function of period $N$ such that $\sum_{a=1}^N f(a) = 0$. Then the coefficient of $\log p$ in the expansion of $L(1,f)$ is given by $\frac{-p}{(p-1)N}\sum_{a=1}^{N/p}f(pa)$. In particular if $L(1,f)$ is zero, then $\sum_{a=1}^{N/p}f(pa) = 0$. 
\end{Cor}
\begin{proof}
In this case, we have $1_p(n) = 1 - v_p(n)$ as $v_p(N) = 1$. We explicitly compute $(V_p^{pr})$ in this case. We set $ M_d $ to be the restricted sum of the arithmetical function $ f_d $ that is $M_d = \sum_{a=1}^{N/d} f_d(a)$. Then the coefficient of $\log p$ in the expansion of $L(1,f)$ is given by the following :
\begin{flalign*}
    \frac{-1}{N}\sum_{d \mid N}M_d(v_p(d) - \frac{1_p(d)}{p-1}) &= \frac{-1}{(p-1)N}\sum_{d \mid N} M_d \big((p-1)v_p(d) - (1 - v_p(d))\big)\\
    &= \frac{-p}{(p-1)N}\sum_{d \mid N}M_d v_p(d) =\frac{-p}{(p-1)N}\sum_{p \mid d \mid N}M_d 
\end{flalign*}
To get the second last equality, we use $\sum_{d \mid N} M_d = \sum_{a=1}^N f(a) = 0$. We can further write 
\[\sum_{p \mid d \mid N} M_d = \sum_{d \mid N/p} M_{pd} = \sum_{a=1}^{N/p} f(pa). \] 
We obtain the desired corollary with the above computation.
\end{proof}

\begin{Rem}
We now connect our vanishing criterion and Property \ref{PROPERTY-U}. Let $f$ be a rational valued function of period $N$. Tracing back our steps in the proof, by \eqref{EQN:S_d-R_d-BEFORE-SPLIT}, \eqref{EQN:REARRANGED-PRIM-PART} we note that 
\begin{flalign}\label{EQN:SPECIAL-VALUE-L1F-DIRECT-SUM}
    L(1,f) &= \sum_{d \mid N} S_d + \langle f_d,\chi_{0,N/d}\rangle R_d \nonumber \\
           &= \sideset{}{'}\sum_{e \mid N} L(1,\mathcal{D}_{e}(e,\sum_{d \mid N/e} \frac{1}{d}\beta_{N/d,e}\Proj_{e}^{N/d}(f_d))) + \sum_{d \mid N} \langle f_d,\chi_{0,N/d}\rangle R_d.
\end{flalign}
where the first sum varies over the divisors of $N$ excluding $1$. Therefore, for some proper divisor $D$ of $N$ if $\eqref{EQN:INIT-COND-DIVISORS}$ is not satisfied, then by Remark \ref{REM:DEF:IMPRIM:REINTEPRET} we have 
\[\mathcal{D}_{N/D}(\frac{N}{D},\sum_{d \mid D} \frac{1}{d}\beta_{N/d,N/D}\Proj_{N/D}^{N/d}(f_d)) \neq 0. \]
Since $D \neq N$, by Remark \ref{REM:DPRIMITIVE-PROPERTY-U} we note that the above arithmetic function satisfies Property \ref{PROPERTY-U}. From here we conclude that $f$ satisfies Property \ref{PROPERTY-U} as the first sum mentioned in \eqref{EQN:SPECIAL-VALUE-L1F-DIRECT-SUM} is a direct sum.  
\end{Rem}
\section{Association to Group Ring}\label{SEC:GROUP-RINGS}
In this section, we associate a Dirichlet type function $f \in \mathbf{F_D}(N)$ with an element $\mathcal{P}(f)$ in the group ring $\mathbb{Q}[\Gal{\mathbb{\mathbb{Q}(\zeta_N)}/\mathbb{Q}}]$ and prove imprimitivity results by finding corresponding annihilator ideals. Our association of the function $f$ with the element $\mathcal{P}(f)$ is strictly not necessary, but the advantage is that it is easy to keep track of the action of the Galois element on the whole of $f$. Throughout, we set $G=\Gal{\mathbb{\mathbb{Q}(\zeta_N)}/\mathbb{Q}}$.

We recall the definition of $\mathcal{P}(f)= \sum_{(a,N)=1}f(a) \sigma_a$. We note that $\mathcal{P}(\sigma_bf) = \sigma_b\mathcal{P}(f)$. This gives a $\mathbb{Q}[G]$-module isomorphism between $\mathbf{F_D}(N)$ and $\mathbb{Q}[G]$. 
\par
We now give a simple description to understand when an arithmetic function is imprimitive $ \bmod{~} N $ and also provide an expression for evaluating $ \Proj_d^N(f) $ in terms of $ f $. Before stating the Proposition, we evaluate Dirichlet characters on $\Gal{\mathbb{Q}(\zeta_N)/\mathbb{Q}}$ instead of $(\mathbb{Z}/N\mathbb{Z})^*$. This is done so that we do not have to explicitly define the automorphism while considering the character decomposition of the action. This is also done in \cite{WASHINGTON}. If $\sigma_a(\zeta_N) = \zeta_N^a$, then we define {$\chi(\sigma_a) = \overline{\chi}(a)$}. The complex conjugation is done so that we have the following : 
\[ f = \sum_{\chi \bmod{~} N}c_\chi \chi \implies \sigma_a f = \sum_{\chi \bmod{~} N} \chi(\sigma_a)c_\chi \chi . \]
\begin{Prop}\label{PROP:POLY-NPRIM-ANNHILATE} 
Let $f$ be a Dirichlet type function of period $N$. 
\begin{enumerate}
    \item For each prime $p$ dividing $N$, let $\tau_p^{(N)}$ be a generator of the group $\Gal{\mathbb{Q}(\zeta_N)/\mathbb{Q}(\zeta_{N/p})}$. \begin{equation}\label{EQN:N-IMPRIMITIVE-POLY-CONDN} 
    f \text{ is }  \text{ imprimitive } \bmod{} N \iff \prod_{p \mid N}(1-\tau_p^{(N)}) \mathcal{P}(f) = 0.
    \end{equation}
    \item For any divisor $d \mid N$, we have 
    \begin{equation}\label{EQN:PROJ-NDF-DESCRIPTION}
	    \mathcal{P}(\Proj_d^N(f)) = \frac{\varphi(d)}{\varphi(N)}\mathcal{P}(f)\big\vert_{\mathbb{Q}(\zeta_d)}.  
    \end{equation}
\end{enumerate}
\end{Prop}
In the above, restricting an element of $\Gal{\mathbb{Q}(\zeta_N)/\mathbb{Q}}$ to $ \mathbb{Q}(\zeta_d) $ gives us an element of $ \Gal{\mathbb{Q}(\zeta_d) / \mathbb{Q}} $. 
\begin{proof}[Proof of Proposition \ref{PROP:POLY-NPRIM-ANNHILATE}]
The proof of the Proposition relies on looking at the character decomposition of the Dirichlet type function $f \in \mathbf{F_D}(N)$. For the function $f$, we write its character decomposition as follows : $f = \sum\limits_{\chi \bmod{~} N} c_\chi \chi$.   
\begin{enumerate}
    \item We have 
    \[(1-\tau_p^{(N)})f = \sum_{\chi \bmod{} N}(1-\chi(\tau_p^{(N)}))c_\chi \chi.\]
    Since $\tau_p^{(N)}$ fixes $\mathbb{Q}(\zeta_{N/p})$ we conclude that (See \cite[Page 20]{WASHINGTON}) :
    \begin{equation*}
        \chi(\tau_p^{(N)}) = 1 \text{ if and only if } p^{v_p(N)} \nmid \cond{\chi}.
    \end{equation*}
    Hence 
    \begin{equation*}
        (1-\tau_p^{(N)})f = \sum_{\substack{\chi \bmod{} N\\ p^{v_p(N)} \mid \cond{\chi} }}(1-\chi(\tau_p^{(N)}))c_\chi \chi.
    \end{equation*}
    Iterating the above equation for all prime divisors $p$ of $N$, we have 
    \begin{equation*}
        \prod_{p \mid N}(1-\tau_p^{(N)})f = \sum_{\substack{\chi \bmod{} N\\ \cond{\chi}=N}}\prod_{p \mid N}(1-\chi(\tau_p^{(N)}))c_\chi \chi.
    \end{equation*}
    Since $f$ is imprimitive $\bmod{~}N$, we have $c_\chi = 0$ for all characters $\chi$ of conductor $N$ by Remark \ref{REM:DEF:IMPRIM:REINTEPRET}. Identifying $f$ with $\mathcal{P}(f)$, we obtain the result. The reverse implication is immediate by the orthogonality of Dirichlet characters $ \bmod{~} N$.
    \item Consider the operator $\beta:=\sum\limits_{\substack{a \in(\Z/N\Z)^*\\a \equiv 1 \bmod{} d}}\sigma_a^{-1}$. For a Dirichlet type function $f$ of period $N$, we have 
    \begin{equation*}
        \beta f = \sum_{\chi \bmod{} N} (\sum_{\substack{a \in(\Z/N\Z)^*\\a \equiv 1 \bmod{} d}}\overline{\chi}(\sigma_a))c_\chi \chi=\frac{\varphi(N)}{\varphi(d)}\sum_{\substack{\chi \bmod{} N\\ \cond{\chi} \mid d}} c_\chi\chi. 
    \end{equation*}
    The second equality above is true as for a given character $\chi$, the required character sum is evaluated as follows :  
    \begin{equation*}
        \sum_{\substack{a \in(\Z/N\Z)^*\\a \equiv 1 \bmod{} d}}\chi(\sigma_a) = 
        \begin{cases}
            \varphi(N)/\varphi(d) &\text{ if } \cond{\chi} \mid d \\
            0 &\text{ else}.
        \end{cases}
    \end{equation*}
    Comparing $\beta f$ with $\Proj_d^N(f)$, we observe that 
    \[\Proj_d^N(f)(n)= \frac{1}{\varphi(N)/\varphi(d)}(\beta f)(n) \text{ whenever } (n,N)=1.\] 
    By choosing a representative $ n $ which is co-prime to $ N $, we can expand the above as : 
    \[  \Proj_d(f)(n) = \frac{\varphi(d)}{\varphi(N)} (\sum_{\substack{b=1\\b \equiv 1 \bmod{} d}}^N \sigma_b^{-1})f(n) =  \frac{\varphi(d)}{\varphi(N)} \sum_{\substack{b=1\\b \equiv 1 \bmod{} d}}^N f(bn).  \]
    We now evaluate $\mathcal{P}(\Proj_d^N(f))$ in terms of $\mathcal{P}(f)$. We have 
    \begin{flalign*}
        \mathcal{P}(\Proj_d^N(f)) &= \frac{\varphi(d)}{\varphi(N)}\sum_{a \in (\Z/d\Z)^*}\sum_{\substack{b \in (\Z/N\Z)^*\\b \equiv 1 \bmod{} d}} f(ab)\sigma_a.\\
				  &= \frac{\varphi(d)}{\varphi(N)} \mathcal{P}(f)\vert_{\mathbb{Q}(\zeta_d)},
    \end{flalign*}
    where in the last step, we restrict each element of $ \Gal{\mathbb{Q}(\zeta_N) / \mathbb{Q}}$ to $ \Gal{\mathbb{Q}(\zeta_d) / \mathbb{Q} } $. %
\end{enumerate}
\end{proof} 
\begin{Rem}
We actually show that there is a one to one correspondence between $\mathbf{F_D^{pr}}(N;N)$ and the ideal $\prod_{p \mid N} (1 - \tau_p^{(N)})$ in $\mathbb{Q}[G]$. More precisely, the map $\prod_{p \mid N} (1 - \tau_p^{(N)})$ maps $\mathbf{F_D}(N)$ onto $\mathbf{F_D^{pr}}(N;N)$. When this map is restricted to $\mathbf{F_D^{pr}}(N;N)$, this gives an isomorphism. To arrive at $(V_D)$ from \eqref{EQN:INIT-COND-DIVISORS}, we apply the above Proposition for evaluating $\Proj$ and using the factor $\prod_{p \mid N} (1-\tau_p^{(N)})$ to get the result. More precisely if $L(1,f) = 0$, then for a given divisor $D$ of $N$ the condition \eqref{EQN:INIT-COND-DIVISORS} is same as saying that $(V_D)$ is satisfied. 
\end{Rem}

For an odd squarefree integer $ N $, we note that the restriction of a generator $ \tau_p^{(N)} $ to $\mathbb{Q}(\zeta_{N/q})$ is also a generator of  $ \Gal{\mathbb{Q}(\zeta_{N/q}) / \mathbb{Q}(\zeta_{N / pq })} $ for a prime divisor $ q \neq p $, that is we can choose generators $\tau_p^{(N)}$ and $\tau_p^{(N/q)}$ satisfying $\tau_p^{(N)}\vert_{\mathbb{Q}(\zeta_{N/q})} = \tau_p^{(N/q)}$. This can be verified as $ \Gal{\mathbb{Q}(\zeta_N) / \mathbb{Q}(\zeta_{N / p })} \cong \Gal{\mathbb{Q}(\zeta_p) / \mathbb{Q}} $. This fact is crucial and will be repeatedly used. For the purpose of simplicity, from now we refer to these generators as $\tau_p$. Therefore for a squarefree integer $N$ having a prime divisor $p$, we have $\tau_p^{-1}(\zeta_p) = \zeta_p^{g_p}$, where $g_p \bmod{~} p$ generates $(\Z/p\Z)^*$ and $\tau_p(\zeta_q) = \zeta_q$ for primes $q \neq p$.

\par
From Proposition \ref{PROP:POLY-NPRIM-ANNHILATE}, we observe that for an integer valued arithmetic function $f$ of period $N$, $\Proj_{N/p}^N(f)$ need not be integer valued. However, under some constraints, we show that (up to a multiple of a zero divisor in $\mathbb{Z}[\Gal{\mathbb{Q}(\zeta_N/p)/\mathbb{Q}}]$) this arithmetic function can be integer valued.
\begin{Prop}\label{PROP:IMPRIMITIVE-INDUCTION}
	let $N$ be an odd squarefree integer. Assume that the Dirichlet type arithmetic function $f$ takes values integer values and is imprimitive $\bmod{~}N$. Then 
	\begin{equation*}
		\prod_{\substack{q \mid N\\q \neq p}}(1-\tau_q) \mathcal{P}(\Proj_{N / p}^N(f)) \in \Z[\Gal{\mathbb{Q}(\zeta_{N / p }) / \mathbb{Q} }],  
	\end{equation*}
	where the product runs over the prime divisors $q$ of $N$ except $p$. In fact, \[\prod\limits_{\substack{q \mid N\\q \neq p}}(1-\tau_q) \mathcal{P}(\Proj_{N / p}^N(f)) = \prod\limits_{\substack{q \mid N\\q \neq p}}(1-\tau_q) \mathcal{P}(h),\]where the arithmetic function $h \in \mathbf{F_D}(N/p)$ is given by $h(n) = f(\widetilde{n})$ where $\widetilde{n}$ is a representative of a residue $\bmod{~}N$ satisfying $\widetilde{n}\equiv 1 \bmod{~}p$ and $\widetilde{n} \equiv n \bmod{~} N/p$.
\end{Prop}

For the proof of the above Proposition we require functions similar to that of $h$, so we set 
\[\mathbf{a_i} : = \prod_{\substack{q \mid N\\q \neq p}}(1-\tau_q)\sum\limits_{j \bmod{~} (N/p)^*} (f(\widetilde{n_j})) \sigma_j,\] where $\widetilde{n_j} \equiv g_p^{-i} \bmod{~} p$ and $\widetilde{n_j} \equiv j^{-1} \bmod{~} N/p$. The element $\sigma_j$ of $\Gal{\mathbb{Q}(\zeta_{N})/\mathbb{Q}}$ denotes the automorphism {$\sigma_j(\zeta_{N/p}) = \zeta_{N/p}^j$} and $\sigma_j$ fixes $\mathbb{Q}(\zeta_p)$. Depending on our choice of fields, we may identify $\sigma_j$ as an element of $\Gal{\mathbb{Q}(\zeta_{N})/\mathbb{Q}}$ or $\Gal{\mathbb{Q}(\zeta_{N/p})/\mathbb{Q}}$. 
\begin{proof}
	Let $ H = \langle \tau_p \rangle $, a subgroup of $ \Gal{\mathbb{Q}(\zeta_N) / \mathbb{Q}} $. Since $f$ is imprimitive $ \bmod{~} N $, $\mathcal{P}(f)$ satisfies \eqref{EQN:N-IMPRIMITIVE-POLY-CONDN}. From the definition of $\mathbf{a_i}$ we have 
	\begin{flalign*}
		 \sum_{i=0}^{p-2} \mathbf{a_i} \tau_p^i&=\prod_{\substack{q \mid N\\q \neq p}}(1-\tau_q)\mathcal{P}(f).\\ 
		                                & = \tau_p\prod_{\substack{q \mid N\\q \neq p}}(1-\tau_q)\mathcal{P}(f) \qquad (\text{ as }f \text{ is imprimitive} \bmod{~}N).\\ 
		                                &= \sum_{i=0}^{p-2} \mathbf{a_i} \tau_p^{i+1}.
	\end{flalign*}
	Therefore $\mathbf{a_i} = \mathbf{a_{i+1}}$ for all $i$. Evaluating the expression as mentioned in \eqref{EQN:PROJ-NDF-DESCRIPTION} we have
	\[
		\prod_{\substack{q \mid N\\q \neq p}}(1-\tau_q) \mathcal{P}(\Proj_{N / p}^{N}(f)) =   \prod_{\substack{q \mid N\\q \neq p}}(1-\tau_q) \frac{1}{p-1}\mathcal{P}(f)\vert_{H=\{1\}} = \frac{1}{p-1} \sum_{i=0}^{p-2} \mathbf{a_i} \tau_p^i\vert_{\tau_p=1} = \mathbf{a_0}.
	\]
	In the last step, we used that $H \cap \langle\tau_q \colon q \neq p \rangle = 1$ and we have $\mathbf{a_i} = \mathbf{a_{i+1}}$ for all $i$. This proves the Proposition as $\mathbf{a_0} \in \Z[\Gal{\mathbb{Q}(\zeta_{N/p})/\mathbb{Q}}]$. The second part of the Proposition is immediate as $\mathbf{a_0} = \mathcal{P}(h)$ (by the definition of $\mathbf{a_i}$).  
\end{proof}
\begin{Rem}
From the above Proposition, we conclude that if a function $f$ having coefficients in a set $S$ is imprimitive $\bmod{~}N$, then we can find a function $h$ of period $N/p$ having coefficients in the same set $S$ such that \[\sideset{}{'}\prod\limits_{q \mid N/p}(1-\tau_q) \mathcal{P}(\Proj_p^N(f)) = \sideset{}{'}\prod\limits_{q \mid N/p}(1-\tau_q) \mathcal{P}(h),\] where the product runs over all the primes $q \neq p$. In the forthcoming section, we apply this observation to functions $f$ taking values in $\{\pm 1\}$. 
\end{Rem}
\begin{Rem}
The proof of the above Proposition also follows through when we take $N= 4p$ for an odd prime $p$. In this case, we have to consider $(1-\tau_4) \Proj_4^N(f)$ and $(1-\tau_p) \Proj_p^N(f)$. 
\end{Rem}
\section{Proof of Theorem \ref{TH:4p} and a counterexample by Tengely}\label{SEC:TH-4P}
In \cite[Theorem 4.8]{MURTY-SIDDHI}, the authors proved the following : Suppose $N \equiv 2 \bmod{~}4$ and the function $f$ is of period $N$ taking values in $\{ \pm 1\}$. Then $L(1,f) \neq 0$ whenever the infinite sum converges. From their argument one can deduce the following result : The coefficient of $\log 2$ in $L(1,f)$ is non-zero owing to parity obstruction. In this article, Theorem \ref{TH:4p} addresses the same question when $N$ is a prime multiple of $4$. This is in some sense the best possible as we have a counter-example \cite{TIJDEMAN} for $N=36$.

In this section, we first proceed with the proof of Theorem \ref{TH:4p}. We apply \eqref{EQN:COND-PRIMES} in a `inductive' manner. We show that $f_p = \pm \chi_{0,4}$ and $ f_4 = \pm \sigma f_2 $. We have two cases for $f_4$ and when $f_4 =\sigma f_2$, some elementary observations on the restricted sums of $ f_d $ for divisors $ d $ of $ 4p $, will give us a contradiction. However the proof is delicate when $f_4 = -\sigma f_2$. We should mention that the earlier methods for proving Erd\"{o}s conjecture do not apply in this case as we have functions $ f $ for which Property \ref{PROPERTY-U} fails. Such a function can be seen in the following example.

\begin{Ex}\label{EXAMPLE:0-UNIT-CONTRIBUTION}
Let $f$ be of period $12$ given by $f_1 = \chi_{0,12}$, $f_2 = -\chi_{0,6}$, $f_3 = \chi_{0,4}, f_4 = -\chi_{0,3}, f_6 = -\chi_{0,2}, f_{12}= -\chi_{0,1}$. We note that $\sum_{a=1}^{12} f(a) =0 $, and naturally $f_i$ are imprimitive $\bmod{~} D$ for any divisor $D \neq 1$ of $12$. Therefore $\Proj_{12/D}^{12/d}f_d = 0$ for all $D \neq 12$. Thus $L(1,f) \in \mathbb{Q}\langle\log 2,\log 3\rangle$.
\end{Ex}

In the following proof, for a divisor $ i $ of $ 4p $, we define the restricted sum of $f_i$ as follows : 
\[ M_i = \sum_{\substack{j=1\\(j,N/i)=1}}^{N/i}f_i(j). \] 
The following observations are used in the proof repeatedly : 
\begin{equation}\label{EQN:TH4P:REPEAT-OBS-PARITY}
	(1-\tau_p) \sum_{i}a_i \tau_p^i \text{ has coefficients in } \{0 ,\pm 2\} \text{ whenever } a_i \in \{ \pm 1 \} \tag{O1}. 
\end{equation}
Suppose $d= 4$ or $d$ is an odd prime. 
\begin{flalign}\label{EQN:TH4P:REPEAT-OBS-PRIMITIVE}
    \text{ If }f \in \mathbf{F}_D(d) \text{ is imprimitive} \bmod{d} \text{ then } f(a) = \text{ constant whenever} (a,d) = 1. \tag{O2}
\end{flalign}
Finally we remark that if the coefficient of $\log p$ is zero in the expansion of $L(1,f)$, then by using $\sum_{a=1}^N f(a) = 0$ and Corollary \ref{COR:LOGP0-CONSEQUENCE} we have 
\begin{equation}\label{EQN:CRUCIAL-LOGP-CONTRIBUTION-SQFREE-DIV} 
	\sum_{\substack{d \mid N\\p \mid d}} M_d = \sum_{\substack{ d \mid N \\ p \nmid d}}M_d = 0. \tag{O3}
\end{equation}

\begin{proof}[Proof of Theorem \ref{TH:4p}]
	Assume that there exists an arithmetic function $f$ taking values in $\{\pm1\}$ for which $L(1,f) = 0$. From Corollary \ref{COR:N-prim-nonvanishing}, we note that $ f_1 $ is imprimitive $\bmod{~}4p$.

	\textbf{Imprimitivity due to valuation arguments : }  By considering $(V_p)$, we have :  
	\[ \left(1-\tau_4 \right) \big((1 - \frac{\sigma_p^{-1}}{p}) \Proj_4^N(f_1) + \frac{1}{p} f_p \big) = 0. \]
	Rearranging this, we get 
	\[ (1-\tau_4) \mathcal{P}(\Proj_4^N(f_1)) + \frac{(1-\tau_4)(\mathcal{P}(f_p)) - \sigma_p^{-1} (1-\tau_4)\mathcal{P}(\Proj_4^N(f_1))}{p} = 0.  \]
	By Proposition \ref{PROP:IMPRIMITIVE-INDUCTION}, we can replace $ (1-\tau_4)\mathcal{P}(\Proj_4^N(f_1)) $ with $ (1-\tau_4) \mathcal{P}(g) $ for some appropriate $ g \in \mathbf{F_D}(4) $ taking values in $ \{\pm 1\} $ at $1$ and $3$. By \eqref{EQN:TH4P:REPEAT-OBS-PARITY}, the coefficients of $ (1-\tau_4) (\mathcal{P}(f_p) - \sigma_p^{-1}\mathcal{P}(g)) $ will lie in $ \{ 0, \pm 2, \pm 4 \} $ and this is not divisible by $  p$ unless it is zero (as p is an odd prime). Therefore, $(1- \tau_4) \mathcal{P}\left(\Proj_4^N(f_1) \right) = 0$ and this implies $(1-\tau _4) \mathcal{P}\left(f_p \right) = 0$. 
	With this, we concluded that 
	\begin{equation}\label{EQN:TH-4p-IMPRIMITIVE-MOD4}
	    \Proj_{4}^N(f_1) \text{ and } f_p  \text { are imprimitive } \bmod{~}4.
	\end{equation}
	\textbf{ Relation between $f_4$ and $f_2$ : } From $(V_4)$ we note that
	\begin{flalign*}
		(1-\tau_p) \big(\mathcal{P}(\Proj_p^N(f_1)) + \frac{1}{2}\left(\mathcal{P}(\Proj_p^{2p}(f_2)) - \sigma_2^{-1}  \mathcal{P}(\Proj_p^{N}(f_1)) \right) \\
		+\frac{1}{4}\left( -\sigma_2^{-1}\mathcal{P}(\Proj_p^{2p}(f_2))+  \mathcal{P}(f_4)   \right) \big) = 0.  
	\end{flalign*}
By Proposition \ref{PROP:IMPRIMITIVE-INDUCTION}, we have $ (1-\tau_p)\mathcal{P}(\Proj_p^N(f_1)) = (1-\tau_p) \mathcal{P}(g) $ for some appropriate $ g \in \mathbf{F_D}(p) $ taking values in $ \{\pm 1\} $. Therefore, by \eqref{EQN:TH4P:REPEAT-OBS-PARITY} the element {$(1-\tau_p)\mathcal{P}(\Proj_p^N(f_1)) \in 2\Z[\tau_p]$} and thus 
\[ (1-\tau_p)\mathcal{P}(\Proj_p^N(f_1)) + 2^{-1}(1-\tau_p)\left(\mathcal{P}(\Proj_p^{2p}(f_2)) - \sigma_2^{-1}  \mathcal{P}(\Proj_p^{N}(f_1)) \right) \big) \in \mathbb{Z}[\tau_p]. \] Therefore the element 
	\[ \frac{1}{4} (1-\tau_p) \left(\mathcal{P}(f_4) - \sigma_2^{-1} \mathcal{P}(\Proj_p^{2p}(f_2)) \right) \in   \mathbb{Z}[\tau_p].\]
	We write $  \mathcal{P}(f_4) - \sigma_2^{-1} \mathcal{P}(\Proj_p^{2p}(f_2)):=2\sum_{i=1}^{p-1}a_i \tau_p^i := \mathcal{P}(g) $. Note that $ a_i \in \{0, \pm 1\}$, as the functions $ f_4,f_2 $ take values in $ \{ \pm 1 \} $. Thus $ (1-\tau_p) \mathcal{P}(g) $ takes values in $ \{0 ,\pm 1, \pm 2 \} $. To ensure that $ (1-\tau_p) \mathcal{P}(g) $ is even valued, we should have $ a_i = a_{i-1} \bmod{~} 2 $. However, this means that $ a_i = 0 $ for all $ i $ or $ a_i \in \{ \pm 1\} $ for all i. If $ a_i = 0 $ for some $ i$ (hence for all $ i  $) then $ f_4 = \sigma_2^{-1}\Proj_{p}^{2p}(f_2) $. If $ a_i \neq 0 $ for all $ i $, then we note that $ f_4(i) \neq \sigma_2^{-1}\Proj_{p}^{2p}(i) $ for all $ i $ co-prime to $ p $ and in this case, we have $ f_4 = -\sigma_2^{-1} \Proj_{p}^{2p}(f_2) $. In particular $ M_4 = \pm M_2$ by  \eqref{EQN:PROJ-NDF-DESCRIPTION}.

    \textbf{Conditions on restricted sums : } We have the following equations consisting of $ M_i $'s arising from $(V_2^{pr})$ and \eqref{EQN:CRUCIAL-LOGP-CONTRIBUTION-SQFREE-DIV} in Remark \ref{COR:LOGP0-CONSEQUENCE} as $v_p(N) = 1$ in this case.  
	\begin{align}
		-M_1 + 2M_4 -M_p + 2M_{4p} &= 0 \label{EQN:TH-4P:LOG2VAL}.\\
		M_p + M_{2p} + M_{4p} &= 0 \label{EQN:TH-4P:SPLITSYSTEMFROM1SYSTEM}. \\
		M_1 + M_2 + M_4 &=0 \label{EQN:TH-4P:124}.
	\end{align}	
	
	From \eqref{EQN:TH-4P:SPLITSYSTEMFROM1SYSTEM} we also note that $ M_{2p} = M_{4p} $ as $M_p = \pm 2$ (note that $ f_p = \pm \chi_{0,4} $ by \eqref{EQN:TH4P:REPEAT-OBS-PRIMITIVE}).  Inserting this in \eqref{EQN:TH-4P:LOG2VAL}, we obtain 
	\begin{equation}\label{EQN:TH-4P:NEWLOG2VAL}
		M_1-2M_4 = -2M_p.
	\end{equation}
	\textbf{Analysis by possibilities on $f_4$ : } 
	\begin{itemize}
	    \item $\mathbf{f_4 = \sigma_2^{-1} \Proj_{p}^{2p}(f_2)}$ : In this case, we have ${M_2 = +M_4}$ and we solve equations \eqref{EQN:TH-4P:124},\eqref{EQN:TH-4P:NEWLOG2VAL}. We find that $|M_4| = 1 $. However this is a contradiction as $f_4 \in \mathbf{F_D}(p)$ takes values in $\{ \pm 1\}$ and consequently $M_4 \equiv 0\bmod{~}2$.
	    \item $\mathbf{f_4 = - \sigma_2^{-1} \Proj_{p}^{2p}(f_2)}$ : In this case, we have ${M_2 = -M_4}$ and substituting this in equation \eqref{EQN:TH-4P:124} we conclude that $M_1 = 0$. Recalling that $f_1$ is imprimitive $\bmod{~}{4p}$ and $\Proj_4^N(f_1)$ is imprimitve $\bmod{~} 4$, by Remark \ref{REM:PROJ} we have 
	    \[ f_1 = \sum_{\substack{\chi \bmod{}{N}\\\text{ cond}(\chi) = p}} \langle f_1,\chi \rangle \chi. \]
	    In particular, $\Proj_p^N(f_1)$ takes values in $\{ \pm 1\}$ at co-prime residue classes $\bmod{~} p$ and $\sum_{a=1}^p\Proj_p^N(f_1)(a) = 0$. Now, on re-writing $(V_4)$ assuming these conditions, by \eqref{EQN:TH4P:REPEAT-OBS-PRIMITIVE} we have : 
	    \[ (1 -\frac{\sigma_2^{-1}}{2}) \Proj_p^N(f_1)(j) + \frac{1}{2}(1-\sigma_2^{-1}) \Proj_p^{2p}(f_2)(j) = c \text{ for all } (j,p) = 1. \]
	    This forces the value of $c$ to be a half integer. On summing over all the values $j$ co-prime to $p$, we have  
	    \[ (p-1)c = \frac{1}{2} \sum_{j=1}^{p-1} \Proj_p^N(f_1)(j) = 0. \]
        This implies $c = 0$, a contradiction to the fact that it is a half integer. 
	\end{itemize}
\end{proof}
\subsection{A counter-example by Sr. Tengely}
It was mentioned in \cite[Pg. 382, Theorem 1]{TIJDEMAN} that using an exhaustive search, Tengely explicitly constructed an arithmetic function $f$ with values in $\{ \pm 1 \}$ and of period $36$ such that $L(1,f) = 0$. A proof using digamma relations was given in \cite{PILERUD}. We provide a simpler explanation using the decomposition of $f$ as Dirichlet type functions. This is in the spirit of Chowla's question \cite{CHOWLA}. We start by recalling the function $f$ : \\ 
\begin{center}
\begin{tabular}{|c|c|c|c|c|c|c|c|c|c|c|c|c|c|c|c|c|c|c|c|c|c|c|c|c|c|c|c||c|c|c|c|c|c|c|c|}
     \hline
     n&1&2&3&4&5&6&7&8&9&10&11&12&13&14&15&16&17&18 \\
     \hline
     f(n)&1&-1&-1&-1&-1&1&1&1&-1&1&-1&-1&1&-1&1&-1&-1&1\\
     \hline
     n&19&20&21&22&23&24&25&26&27&28&29&30&31&32&33&34&35&36\\
     \hline
     f(n)&1&1&-1&1&-1&-1&1&-1&-1&-1&-1&1&1&1&1&1&-1&1\\ 
     \hline
\end{tabular}
\end{center}
\par
If we set $\Psi_3$ as the non-trivial character $\bmod{~} 3$, $\Psi_6$ as the non-trivial character $\bmod~{6}$ and $\chi_{0,N}$ as the trivial character $\bmod N$, then from the above values, we observe that 
\[ f_1 = \Psi_6, \quad f_2 = -\Psi_6, \quad f_3 = -\Psi_6, \quad f_4 = -\Psi_3, \quad \]
\[f_6 = \chi_{0,6}, \quad f_9 = -\chi_{0,4}, \quad f_{12} = -\chi_{0,3}, \quad f_{18} = \chi_{0,2}, \quad f_{36} = 1. \]
Now we have 
\[ L(s,f) = \sum_{d \mid 36} \frac{1}{d^s}L(s,f_d) = \sum_{d \in S_1} \frac{1}{d^s}L(s,\chi_{0,d}) + \sum_{d \in S_2} \frac{1}{d^s}L(s,f_d).  \]
Here $S_1 = \{6,9,12,18,36 \} \text{ and } S_2 = \{ 1,2,3,4 \}$. So it remains to understand the relationship between the Dirichlet series $L(s,\chi_{0,N}), L(s,\Psi_6)$ and $L(s,\Psi_3)$. We note that 
\[ L(s,\chi_{0,N}) = \prod_{p \mid N} \big(1-\frac{1}{p^s} \big)\zeta(s), \qquad L(s,\Psi_6) = \big(1-\frac{\Psi_3(2)}{2^s} \big)L(s,\Psi_3) = \big(1+\frac{1}{2^s} \big)L(s,\Psi_3). \]
The second equality holds true as $\Psi_6$ is induced from $\Psi_3$ and therefore the Dirichlet series associated to these characters only differ by an Euler factor. Substituting the above expressions and on expanding the Dirichlet polynomials we get 
\[ L(s,f) = \big(\frac{1}{6^s} - \frac{1}{9^s} - \frac{2}{12^s} + \frac{1}{18^s} + \frac{2}{36^s} \big)\zeta(s) + \big(1 - \frac{1}{3^s} - \frac{2}{4^s} - \frac{1}{6^s} \big)L(s,\Psi_3). \]
To observe that $L(1,f) = 0$, it suffices to analyse the nature of zeroes of the Dirichlet polynomials at $s=1$. Owing to the pole of $\zeta(s)$ at $s=1$, we have to check that the Dirichlet polynomial $\big(\frac{1}{6^s} - \frac{1}{9^s} - \frac{2}{12^s} + \frac{1}{18^s} + \frac{2}{36^s} \big)$ has a zero of order two at $s=1$. We observe that\footnote{It is enough to see the Dirichlet polynomial and its derivative vanishes at $s=1$ and this can be easily verified. We write it in this manner as these expressions might be useful if we want to determine higher order zeros; computing double derivatives won't help us as it is a non-linear polynomial expression in logarithms of prime numbers}
\[ \big(\frac{1}{6^s} - \frac{1}{9^s} - \frac{2}{12^s} + \frac{1}{18^s} + \frac{2}{36^s} \big) = \frac{1}{6^s}\big(1-\frac{2}{2^s} \big)\big(1-\frac{3}{3^s} \big) - \frac{1}{9^s} \big(1-\frac{2}{2^s}\big)^2, \]
thereby proving that the Dirichlet polynomial has a zero of order two at $s=1$. On the other hand $L(s,\Psi_3)$ is an entire function and $\big(1 - \frac{1}{3^s} - \frac{2}{4^s} - \frac{1}{6^s} \big)$ vanishes at $s=1$. Therefore $L(1,f) = 0$. 
\section{A symmetry involving $\Proj$ operators and vanishing of $L(1,f)$}\label{SEC:ERD-INDUCTION}
    Let $ f $ be periodic of period $ N $ and $ g $ be periodic of period $dN$. We say that $ f $ is a component of $ g $ if $ g(dn)=f(n) $ for all natural numbers $ n $. We also say that the Erd\"{o}s conjecture is true for a function $ f $ by \eqref{EQN:COND-DIVISORS:INTERMS-GROUPRING} if one of the conditions in \eqref{EQN:COND-DIVISORS:INTERMS-GROUPRING} is not satisfied. The proof of Theorem \ref{TH:ERD-INDUCTION} has two components, namely the choice of prime and we also require a `descent' process. The key observation is that if for two integer valued functions $ F,G $ bounded by an absolute constant $ C $ independent of $ p $, if we have 
	\begin{equation}\label{EQN:IND-PROCESS}
	   (p - \sigma_p) \mathcal{P}(F) + \mathcal{P}(G) = 0, 
	\end{equation}
	then for $ p $ large enough compared to $ C $, the functions $ F $ and $ G $ has to be identically zero. In this proof, we assume that $L(1,g)= 0 $ and therefore the components of $g$ satisfy $(V_D)$ for all divisors $D$ of $pN$. Splitting $g_d$ accordingly and choosing $D$ co-prime to $p$, we get functions $F_D$ and $G_D$ such that they satisfy \eqref{EQN:IND-PROCESS}. $\mathcal{P}(F_D)$ and $\mathcal{P}(G_D)$ will be certain linear combinations of elements in $ \mathbb{Z}[\Gal{\mathbb{Q}(\zeta_d) / \mathbb{Q}}]$, and $G_D$ will consist only of elements $\Proj_{N/D}^{N/d}(f_d)$ and some other factor. We conclude that $G_{D}$ is imprimitive $ \bmod{~} N/D $ for $ p \nmid N/D $ owing to the above observation. Our final choice of $ C $ will depend on $ N $ and this is because for every $D$ dividing $N$, we are finding $C_D$.  

\par
We also mention that the function $ \Proj_{d} $ is well behaved under Galois action i.e. if $ {\sigma \in \Gal{\mathbb{Q}(\zeta_N) / \mathbb{Q}} }$ and $ f \in \mathbf{F_D}(N) $, then $ {\ Proj_{D}^N(\sigma f) = \sigma\vert_{\mathbb{Q}(\zeta_D)} \Proj_{D}^N(f)}$. This can be seen by Definition \ref{DEFN:INTRO-PROJ}.

To avoid cluttering of notations, we use the following hyphenated product : 
\[\sideset{}{'}\prod_{q\mid N/d} (1 - \frac{\sigma_q}{q}) = \prod_{\substack{q \mid N/d\\q \nmid N/D}}(1 - \frac{\sigma_q}{q}). \]
Here $D$ is chosen `appropriately', and the elements $(1-q^{-1}\sigma_q)$ belong to $\mathbb{Q} [\Gal{\mathbb{Q}(\zeta_{N/D})/\mathbb{Q}}]$.   
\begin{proof}[Proof of Theorem \ref{TH:ERD-INDUCTION}]
	Let $ g $ be an integer valued function of period $ pN $ satisfying $ L(1,g) = 0 $ and let $ f $ be a component of $ g $ of period $ N $. By Theorem \ref{TH:NEW-NECESSARY-SUFFCIENT-VANISHING-L1F}, we note that $ g $ satisfies \eqref{EQN:COND-DIVISORS:INTERMS-GROUPRING} for every divisor $ D $ of $ N $ and \eqref{EQN:COND-PRIMES} for every prime $ p $ dividing $ N $.

	\textbf{Imprimitivity and Integrality} : Consider $ (V_D) $ for $ p \nmid D $. We have the following : 
	\[ \prod_{q \mid pN/D}(1-\tau_q) \big(\sum_{d \mid D}\frac{1}{d} \prod_{\substack{ q \mid pN/d\\q \nmid pN/D}} (1 -\frac{\sigma_q^{-1}}{q}) \mathcal{P}(\Proj_{pN/D}^{pN/d}(g_d))\big) = 0.  \]
	We define function $h_D$ as follows : 
	$$ h_D := \varphi(D)D  \sum\limits_{d \mid D} \frac{1}{d} \sideset{}{'}\prod\limits_{q \mid pN/d} (1 -\frac{\sigma_q^{-1}}{q}) (\Proj_{pN/D}^{pN/d}(g_d))\big).$$ 
	The function $h_D$ is integer valued and by Proposition \ref{PROP:POLY-NPRIM-ANNHILATE}, $h_D$ is imprimitive $\bmod{~} pN/D$.

	Therefore by Proposition \ref{PROP:IMPRIMITIVE-INDUCTION} we have,
	\begin{equation}\label{EQN:TH-INDN-PROP-IMPACT-hD}
	   \prod_{\substack{q \mid pN/D\\q \neq p}} (1-\tau_q) \mathcal{P}(\Proj_{N/D}^{pN/D}h_D) \in \mathbb{Z}[ \Gal{\mathbb{Q}(\zeta_{N/D}) / \mathbb{Q} }]. 
	\end{equation}
	Moreover, note that we have 
	\begin{equation}\label{EQN:TH-INDN-PROP-IMPACT-BOUND-ON-hD}
	    \prod_{\substack{q \mid pN/D\\q \neq p}} (1-\tau_q) \mathcal{P}(\Proj_{N/D}^{pN/D}h_D) = \prod_{\substack{q \mid pN/D\\q \neq p}} (1-\tau_q) \mathcal{P}(H_D),
	\end{equation}
	where the arithmetic function $H_D(n)$ is given by $H_D(n) = h_D(\widetilde{n})$, where $\widetilde{n}$ is a representative of a residue $\bmod{~}pN/D$ satisfying $\widetilde{n}\equiv 1 \bmod{~}p$ and $\widetilde{n} \equiv n \bmod{~} N/D$.

	\textbf{Arriving at \eqref{EQN:IND-PROCESS} for divisors $D$ of $N$:} Now, we consider $ (V_{pD}) $ for the same divisor $ D $ of $ N $.
	\begin{flalign*}
		\prod_{q \mid N/D}(1-\tau_q) \big(\sum_{d \mid pD}\frac{1}{d} \sideset{}{'}\prod_{q \mid N/d} (1 -\frac{\sigma_q^{-1}}{q}) \mathcal{P}(\Proj_{N/D}^{N/d}(g_d))\big) &= 0. \\
		\prod_{q \mid N/D}(1-\tau_q) \left(1-\frac{\sigma_p^{-1}}{p} \right)  \big(\sum_{d \mid D}\frac{1}{d} \sideset{}{'}\prod_{q \mid N/d} (1 -\frac{\sigma_q^{-1}}{q}) \mathcal{P}(\Proj_{N/D}^{pN/d})(g_d)\big) &~~ + \\
		\frac{1}{p} \big(\prod_{q \mid N/D}(1-\tau_q) \big(\sum_{d \mid D}\frac{1}{d} \sideset{}{'}\prod_{q \mid N/d} (1 -\frac{\sigma_q^{-1}}{q}) \mathcal{P}(\Proj_{N/D}^{N/d}(f_d))\big) &= 0.  
	\end{flalign*}
	Multiplying the above equation by $ pD\varphi(D) $, we find that the denominators in the above equation are cancelled and writing $ \Proj_{N/D}^{pN/d} = \Proj_{ N/D}^{pN/D} \circ \Proj_{pN/D}^{pN/d} $ for $ p \nmid d $, we obtain :
	\begin{flalign*}
	    \prod_{q \mid N/D}(1-\tau_q) \big(\left(p-{\sigma_p^{-1}} \right)  \mathcal{P}(\Proj_{N/D}^{pN/D}(h_D))\big) + \\
	    D\varphi(D) \big(\prod_{q \mid N/D}(1-\tau_q) \big(\sum_{d \mid D}\frac{1}{d} \sideset{}{'}\prod_{q \mid N/d} (\frac{q-\sigma_q^{-1}}{q}) \mathcal{P}(\Proj_{N/D}^{N/d}(f_d))\big)\big) = 0 .
	\end{flalign*}
	Note that the following elements are in $ \mathbb{Z} [ \Gal{ \mathbb{Q}(\zeta_{N / D })/ \mathbb{Q}}] $ : 
	\begin{itemize}
	    \item $\mathcal{P}(F_D): = \prod\limits_{q \mid N/D}(1-\tau_q) \mathcal{P}(\Proj_{N/D}^{pN/D}(h_D)) $ (By  \eqref{EQN:TH-INDN-PROP-IMPACT-hD}).
	    \item $\mathcal{P}(G_D):= D\varphi(D) \big(\prod\limits_{q \mid N/D}(1-\tau_q) \big(\sum\limits_{d \mid D}\frac{1}{d} \sideset{}{'}\prod\limits_{\substack{ q \mid N/d}} (\frac{q-\sigma_q^{-1}}{q}) \mathcal{P}(\Proj_{N/D}^{N/d}(f_d))\big)\big)$.
	\end{itemize}
	\textbf{Upper Bounds : } We now compute the upper bound of the values of the functions $ F_D $ and $ G_D $ respectively. By $\max{f}$, we denote the maximum value of the Dirichlet type function $f$. We require the following classical bounds :
	\[ q+1 \le 2(q-1), \qquad \max{\Proj^{N/d}_{N/D}(g_d)} \le 1. \]
	
	For the second inequality above, we use Proposition \ref{PROP:POLY-NPRIM-ANNHILATE}. This along with estimates \[\displaystyle{|\max((1+\sigma)(r))| \le 2 |\max(r)|}\] 
	applied to \eqref{EQN:TH-INDN-PROP-IMPACT-BOUND-ON-hD} (note that $F_D = \prod_{q \mid N/D}(1-\tau_q) H_D$) will give
	\begin{flalign*}
	|\max{F_D}| \le 2^{\omega(N/D)} |\max{h_D}| &\le 2^{\omega(N/D)} \varphi(D) \sum_{d \mid D} \prod\limits_{\substack{q \mid d\\q \text{ prime }}} (q+1). \\&\le 2^{\omega(N/D)} \varphi(D) \sum_{d \mid D} 2^{\omega(d)} \varphi(d) \le 2^{\omega(N)} \varphi(D) \prod\limits_{\substack{q \mid D\\q \text{ prime }}}(q^2-2q+2).
	\end{flalign*}
	The same bound also holds for $G_D$. Now we set $p$ to be a prime divisor greater than {$2^{\omega(N)}\varphi(N)\prod\limits_{\substack{q \mid N\\q \text{ prime }}}(q^2-2q+2)$.} 
	we have $F_D = G_D = 0$ for all divisors $D \neq 1$ dividing $N$. 
	
	Since $G_D = 0$ we have 
	\[\big(\prod_{q \mid N/D}(1-\tau_q) \big(\sum_{d \mid D}\frac{1}{d} \sideset{}{'}\prod_{q \mid N/d} (\frac{q-\sigma_q^{-1}}{q}) \Proj_{N/D}^{N/d}(f_d)\big)\big)  = 0.\]
	Therefore the arithmetic function $f$ satisfies $(V_D)$ for all divisors $D$ of $N$ except for $D=N$. $f$ also satisfies $(V_N)$ as $\sum_{a=1}^Nf(a) = 0$ (from the assumption on the function $f$). Therefore if $L(1,g) = 0$ for a function $g$ of period $pN$ with $p$ greater than $2^{\omega(N)}\varphi(N)\prod\limits_{\substack{q \mid N\\q \text{ prime }}}(q^2-2q+2)$, then $f$ doesn't satisfy Property \ref{PROPERTY-U} contrary to our assumption on $f$. This proves the theorem. 
\end{proof}
\begin{proof}[Proof of Corollary \ref{COR:NEW-FAMILY-ERDOS}]
    Let $f$ be an Erd\"{o}s function of period $N$. Let $h(n):=f(pn)$ for all numbers $n$ and let $M = \sum_{a=1}^{N/p} h(n)$. If $M$ is non-zero, then by Corollary \ref{COR:LOGP0-CONSEQUENCE} $L(1,f)$ is non-zero. It remains to check what happens when $M = 0$. Note that $h$ is an Erd\"{o}s of period $N$. By following the proof in \cite[Theorem 7]{RAM-SARADHA}, we observe that $h$ satisfies Property \ref{PROPERTY-U} as the coefficient of $\pi$ in $L(1,h)$ is non-zero. Now, by Theorem \ref{TH:ERD-INDUCTION}, $L(1,f)$ cannot be zero.
\end{proof}

\section{Vanishing at $k>1$}\label{SEC:CHOWLA-MILNOR}
In this section, we discuss the vanishing of Dirichlet series $L(s,f)$ for periodic functions $f$ at positive integers greater than one. We begin by noting that Conjecture \ref{CONJ:CHOW-MIL} is the same as saying that the map $\textbf{Ev}_k : \mathbf{F_D}(N) \to \C$ assigning $f$ to $L(k,f)$ is injective for $k>1$.  This conjecture is of the same nature of \eqref{EQN:KEY-POINT}, except that we are evaluating at $s=k$. Therefore the proofs of Lemma \ref{LEM:PRIMPART-ALMOST-SAME-PROJPART}, Theorem \ref{TH:NEW-NECESSARY-SUFFCIENT-VANISHING-L1F} will almost follow through with some minor modifications. We list the modifications below : 
\begin{enumerate}[label=6.\arabic*.]

    \item The proof of Lemma \ref{LEM:PRIMPART-ALMOST-SAME-PROJPART} will follow verbatim if we replace $\beta$ with 
    \[\beta_k:= \prod_{\substack{p \mid N\\p \nmid d}} \big(1 - \frac{\sigma_p}{p^k} \big). \]
    \item \label{MODIFICATION:PRIM-SEP-ATK} The proof of Corollary \ref{COR:SAME-DPRIM-VAL} will follow verbatim when we replace $W_d^{pr}(N)$ with $W_d^{pr,k}(N)$ given by 
    \[ W_d^{pr,k}(N): = \{ L(k,f) ~|~f \in \mathbf{F_D^{pr}}(N;d) \}. \]
    \item The analogy for Theorem \ref{TH:NEW-NECESSARY-SUFFCIENT-VANISHING-L1F} for $k>1$ can be treated simply by computing $S_d$ for divisors $d$ of $N$. At appropriate places, we have to replace $d$ by $d^k$. For instance, we have to replace $S_d$ by $S_d^{(k)}$ given by the following : 
    \[S_d^{(k)}= \frac{1}{d^k}\sum_{\chi \bmod{} N/d} \langle f,\chi \rangle L(k,\chi). \]
\end{enumerate}
After making these modifications, we obtain the following theorem assuming Conjecture \ref{CONJ:CHOW-MIL}: 
\begin{Th}\label{TH:NEW-NECESSARY-SUFFCIENT-VANISHING-LkF}
Let $f$ be a rational valued arithmetic function of period $N$. Then $L(k,f) = 0 $ if and only if the following conditions are satisfied : 
\newline
For each divisor $D \neq 1$ of $N$, the arithmetic function :
\begin{equation}\label{EQN:INIT-COND-DIVISORS-LKF}
	\sum_{d \mid D} \frac{1}{d^k} \prod_{\substack{p \mid N/d\\p \nmid N/D}}(1-\frac{\Frob{p}^{-1}}{p^k}) \Proj_{N/D}^{N/d}(f_{d}) \text{ is } \text{imprimitive} \bmod{~} N/D. 
\end{equation}
\end{Th}
We recall an important computation done recently in \cite[Lemma 3.2]{ABHISHEK-SIDDHI}\footnote{The notations used in this work are different as compared to \cite{ABHISHEK-SIDDHI}}. Let $f$ be an arithmetic function of period $N$ such that $L(k,f) = 0$. For a divisor $d$ of $N$, we define a transform $T_d^{(k)}$ from a vector space of functions of period $N$ to a vector space of functions of period $N/d$. If $f$ is an arithmetic function of period $N$, then $T_d^{(k)}(f)$ is defined as : 
\[T_d^{(k)}(f)(n) = \sum_{m \in \mathcal{M}(N)}\frac{f(dmn)}{m^k}. \] 
\begin{Lemma}\label{LEM:EXPLICIT-COMPUTATION-LKF-SPLIT}
Let $N$ be squarefree and $\widetilde{f_d^{(k)}} := T_d^{(k)}(f)\chi_{0,N/d} \in F_D(N/d)$ where $\chi_{0,N/d}$ denotes the principal character mod $N/d$. Then
\[ L(s,f) = \sum_{d \mid N} \mu(d) \prod_{p \mid d} \left(1-\frac{p^k}{p^s} \right) \, L(s,\widetilde{f_d^{(k)}}),  \]
where $d$ varies over the divisors of $N$ excluding $1$.
\end{Lemma}
We now proceed with the proof of Theorem \ref{TH:CHARACTERISATION-2PRIMES}
\begin{proof}[Proof of Theorem \ref{TH:CHARACTERISATION-2PRIMES}]
We first construct a $\mathbb{Q}$-vector space\footnote{It is strictly not necessary to introduce this vector space but we do it to avoid some tedious computations towards the end.} : 
\[ V_{pq}^{k,l} := \mathbb{Q}\langle \big(1-\frac{p^k}{p^s}\big)\big(1-\frac{q^l}{q^s}\big)\zeta(s), \big(1-\frac{q^k}{q^s}\big)\big(1-\frac{p^l}{p^s}\big)\zeta(s) \rangle. \]
Note that this vector space is of dimension $2$, as the Dirichlet polynomials are linearly independent over $\mathbb{Q}$. We shall show that the set of arithmetic functions of period $q$ satisfying $L(k,f) = L(l,f) = 0$ is a $\mathbb{Q}$ vector space of dimension $2$ and with this its immediate that $L(s,f) \in V_{pq}^{k,l}$. 

For $N = pq$, we use Lemma \ref{LEM:EXPLICIT-COMPUTATION-LKF-SPLIT} assuming that $L(k,f) = 0$. We then have, 
\begin{equation}\label{EQN:DECOMPOSITION-LSERIES-2PRIMES-CHOWLA-MILNOR}
   L(s,f) = \big(1-\frac{p^k}{p^s}\big)L(s,\widetilde{f_p^{(k)}}) + \big(1-\frac{q^k}{q^s}\big)L(s,\widetilde{f_q^{(k)}}) + \big(1-\frac{pq^k}{pq^s}\big)L(s,\widetilde{f_{pq}^{(k)}}). 
\end{equation}
Since $L(l,f) = 0$, substituting $l$ in the above equation, we conclude that 
\begin{equation}\label{EQN:SPECIALISE-AT-DECOMPOSITION-2PRIMES}
  \big(1-p^{k-l}\big)L(l,\widetilde{f_p^{(k)}}) + \big(1-q^{k-l}\big)L(l,\widetilde{f_q^{(k)}}) + \big(1-pq^{k-l}\big)L(l,\widetilde{f_{pq}^{(k)}}) = 0.  
\end{equation}

By Conjecture \ref{CONJ:CHOW-MIL} and \ref{MODIFICATION:PRIM-SEP-ATK} we note that the $\mathbb{Q}$ vector spaces $W_d^{pr,l}(d)$ as $d$ runs over the divisors of $N$ are direct sums. Therefore ${\widetilde{f_p^{(k)}}}$ and ${\widetilde{f_p^{(k)}}}$ are imprimitive $\bmod{~} q$ and imprimitive $\bmod{~} p$ respectively. In particular $\widetilde{f_d^{(k)}}(m) = \widetilde{f_d^{(k)}}(n)$ where $(m,d) = (n,d) = 1$ and $d$ is either $p$ or $q$. We proceed to prove that this implies $f(dm) = f(dn)$ whenever $(m,d) = 1$.

Let $C_d = \max\{ |f(dm) - f(dn)| ~:~ (m,d)=1=(n,d) \}$. Let $a,b$ be co-prime to $d$ and such that $f(da) - f(db) = C_d$. We have 
\begin{flalign*}
    \sum_{m \in \mathcal{M}(pq)} \frac{f(dam)}{m^k} &= \sum_{m \in \mathcal{M}(pq)} \frac{f(dbm)}{m^k} \\
    f(da) - f(db) &= \sum_{\substack{m \in \mathcal{M}(pq)\\m \neq 1}} \frac{f(dbm)- f(dam)}{m^k} 
\end{flalign*}
Taking absolute value of both sides, we note that 
\[ C_d \le C_d ((1-\frac{1}{p^k})^{-1} (1 - \frac{1}{q^k})^{-1} - 1). \]
and this implies that 
\[ 0 \le 2(1-\frac{1}{p^k}) (1 - \frac{1}{q^k}) C_d \le C_d. \]
However note that 
\[ 2(1-\frac{1}{p^k}) (1 - \frac{1}{q^k}) \ge 2(1-\frac{1}{2^2}) (1 - \frac{1}{3^2}) > 1.\]
We therefore conclude that $C_d = 0$ and consequently $f(da) = f(db)$ whenever $(a,d) = (b,d) = 1$ and $d \in \{ p, q\}$. 

From the definition of $\widetilde{f_{N/d}^{(k)}}$ we have : 
\[ L(s,\widetilde{f_{N/d}^{(k)}}) = \sum_{\substack{n=1\\(n,d)=1}}^\infty \frac{\widetilde{f_{N/d}^{(k)}}(n)}{n^s} = \widetilde{f_p^{(k)}(1)} \big(1-\frac{d}{d^s} \big)\zeta(s) \text{ and } L(s,\widetilde{f_{pq}^{(k)}}) = \sum_{n=1}^\infty \frac{f_{pq}^{(k)}(n)}{n^s} = \widetilde{f_{pq}^{(k)}(1)}\zeta(s).\]
Combining the above observations, we have
\[ L(s,f) =\bigg(\widetilde{f_p^{(k)}(1)}\big(1-\frac{p^k}{p^s}\big)\big(1-\frac{1}{q^s}\big) + \widetilde{f_q^{(k)}(1)}\big(1-\frac{p^k}{p^s}\big)\big(1-\frac{1}{q^s}\big) + \widetilde{f_{pq}^{(k)}(1)}\big(1-\frac{pq^k}{pq^s}\big)\bigg) \zeta(s). \]
Since $L(l,f) = 0$, by \eqref{EQN:SPECIALISE-AT-DECOMPOSITION-2PRIMES}, this translates to a linear relation among the numbers $\widetilde{f_p^{(k)}}(1), \widetilde{f_q^{(k)}}(1)$ and $ \widetilde{f_{pq}^{(k)}}(1)$. and this in turn gives us a $\mathbb{Q}$ relation between $f(p), f(q), f(pq)$ as $f_p^{(k)}(1)$ and $f_q^{(k)}(1)$ are $\mathbb{Q}$ linear combinations of $f(p),f(q),f(pq)$. We therefore note that $L(s,f)$ belongs to a vector space of dimension $2$ if $L(k,f) = L(l,f) = 0$. This proves the theorem. 
\end{proof}
\section{Some additional comments}\label{SEC:CONCLUDING-REMARKS}
In this section, we shall discuss the approaches about the proofs, drawbacks, and some further directions using our theorem. 
\begin{itemize}
    \item We first mention the pros and cons of Theorem \ref{TH:NEW-NECESSARY-SUFFCIENT-VANISHING-L1F}. Like the theorem of Okada, Theorem \ref{TH:NEW-NECESSARY-SUFFCIENT-VANISHING-L1F} it is applicable for rational valued functions of a given period (say $N$). The advantage compared to the criterion mentioned by Okada is that the question in some sense is reduced to $d(N) + \omega(N)$ equations (modulo the `imprimitivity' which can be verified by Proposition \ref{PROP:POLY-NPRIM-ANNHILATE}) and provides somewhat of a conceptual framework, which was absent earlier. However this theorem doesn't give a new proof for Conjecture \ref{CONJ:ERDOS} for $q \equiv 3 \bmod{~} 4$. The article \cite{RAM-SARADHA} inspired a new direction for tackling this conjecture where one has to construct a ``suitable unit'' (say $u$) in $\Z[\zeta_q]$ and prove that the component of $L(1,f)$ in $\log u$ is non-zero by arithmetic considerations. In this viewpoint, it is enough to look at the solution of a single equation rather than show that the solutions of a system of equations do not vanish simultaneously. Some prudence should be exercised while approaching the problem from this viewpoint for the following reason : As seen in Example \ref{EXAMPLE:0-UNIT-CONTRIBUTION}, there are even numbers where it can be seen that $L(1,f) \in \mathbb{Q}\langle\log p : p \mid N \rangle$ for functions $f$ taking values in $\{\pm 1\} $. We see that the units do not contribute in this case and there may be odd natural numbers for which a similar situation can happen. We describe this situation by posing the following question : Is there an odd natural number $N$ for which the following holds : 
    $$\sum_{\substack{d \mid N}} a_d \varphi(d) = 0 \text{ with } a_1 = 0 \text{ and } a_d \in \{ \pm 1 \} \text{ for } d \neq 1.$$
    If this happens, then we can take $f_d = a_{N/d} \chi_{0,N/d}$ for all divisors $d$ of $N$ and we note that $(V_D)$ is satisfied for all divisors $D$ of $N$. Therefore $L(1,f) \in \mathbb{Q}\langle\log p~|~p \mid N\rangle$. When this situation occurs, it is better to write $L(s,f) = F(s)\zeta(s)$ where $F(s)$ is a Dirichlet polynomial, and show that $F(s)$ doesn't have a zero of order two at $s=1$ without appealing to Baker's theorem (if possible).
    \item It is possible to use periodic functions instead of Dirichlet characters and modify the statement of Theorem \ref{TH:NEW-NECESSARY-SUFFCIENT-VANISHING-L1F}. When we do this change, we have to ensure that the functions are linearly independent. We should count the number of ``independent periodic functions'' for a given divisor $d$ of $N$ and Theorem \ref{TH:NEW-NECESSARY-SUFFCIENT-VANISHING-L1F} states that there cannot be too many linearly independent functions if $L(1,f)$ has to vanish.  
    \item Corollary \ref{COR:N-prim-nonvanishing} highlighted a new question of checking if an arithmetic function $f$ is primitive $\bmod{~} N$ to check the vanishing of $L(s,f)$ at $s=1$. The primitivity can be checked with the help of Proposition \ref{PROP:POLY-NPRIM-ANNHILATE}. In fact counting imprimitive functions taking values in a finite set is a question of independent interest. Suppose we say that $f \in \mathbf{F_D}(N)$ is primitive $\bmod{~} d$ for a divisor $d$ of $N$ if there exists a character $\chi$ of period $N$ induced by a primitive character of conductor $d$ such that $\langle f,\chi\rangle \neq 0$ and imprimitive $\bmod{~}d$ otherwise. Now we would like to study the distribution of functions taking value in $\{ \pm 1\}$ which are imprimitive $\bmod{~}d$ for divisors $d$ of $N$ lying in a finite set. In view of Proposition \ref{PROP:POLY-NPRIM-ANNHILATE}, we are curious if elementary sieve theoretic methods can be applied in this context, as this seems to be a question related to inclusion exclusion principle.
    \item It is possible to improve the lower bound of the prime $p$ mentioned in Theorem \ref{TH:ERD-INDUCTION} by looking at the group ring isomorphism mentioned in \cite{AYOUB} along with elementary algebraic number theory. However, we would like to see if there are further applications and address the situation for functions $f$ not satisfying Property \ref{PROPERTY-U}. We relegate this to a future work.
\end{itemize}
\section{Acknowledgements}
The author thanks M. Ram Murty, Siddhi Pathak and Brad Rodgers for their comments in the preliminary draft of the paper. 
\bibliographystyle{plain}

\end{document}